\documentclass[11pt,reqno]{amsart}
\usepackage{a4wide}

\numberwithin{equation}{section}\numberwithin{equation}{section}
\usepackage{mathrsfs}
\usepackage{amsfonts}
\usepackage{amsmath}
\usepackage{stmaryrd}
\usepackage{amssymb}
\usepackage{amsthm}
\usepackage{mathrsfs}
\usepackage{url}
\usepackage{amsfonts}
\usepackage{amscd}
\usepackage{indentfirst}
\usepackage{enumerate}
\usepackage{amsmath,amsfonts,amssymb,amsthm}
\usepackage{amsmath,amssymb,amsthm,amscd}
\usepackage{graphicx,mathrsfs}
\usepackage{appendix}

\usepackage[numbers,sort&compress]{natbib}

\usepackage{esint}
\usepackage{graphicx}
\usepackage[dvipsnames]{xcolor}
\usepackage[colorlinks,linkcolor=blue, citecolor=blue]{hyperref}
\usepackage[latin1]{inputenc}

\newtheorem{teo}{Theorem}[section]
\newtheorem{lemma}[teo]{Lemma}
\newtheorem{prop}[teo]{Proposition}
\newtheorem{cor}[teo]{Corollary}
\newtheorem{rem}[teo]{Remark}

\newtheorem{ese}[teo]{Example}
\numberwithin{equation}{section}

\newcommand\R{\mathbb R}

\newcommand\mbb\mathbb
\newcommand\mbf\mathbf
\newcommand\mcal\mathcal
\newcommand\mfrak\mathfrak
\newcommand\mrm\mathrm
\newcommand\msf\mathsf
\renewcommand\a\alpha
\renewcommand\b\beta
\newcommand\g\gamma
\newcommand\G\Gamma
\renewcommand\d\delta
\newcommand\D\Delta
\newcommand\e\varepsilon
\newcommand\z\zeta
\renewcommand\t\theta
\newcommand\Th\Theta
\newcommand\la\lambda
\newcommand\La\Lambda
\newcommand\s\sigma
\newcommand\si\varsigma
\newcommand\Si\Sigma
\newcommand\ups\upsilon
\newcommand\U\Upsilon
\newcommand\ph\varphi
\renewcommand\o\omega
\renewcommand\O\Omega
\newcommand\wt\widetilde
\newcommand\wh\widehat
\newcommand\ol\overline
\newcommand\ul\underline
\newcommand\mr\mathring
\newcommand\ub\underbrace
\newcommand\pa\partial
\newcommand\n\nabla
\newcommand\fa\forall
\newcommand\ex\exists
\newcommand\es\emptyset
\newcommand\wk\rightharpoonup
\newcommand\inc\hookrightarrow
\newcommand\linf\varliminf
\newcommand\lsup\varlimsup
\newcommand\os\overset
\newcommand\us\underset
\newcommand\sr\stackrel
\newcommand\Ot\Leftarrow
\newcommand\To\Rightarrow
\newcommand\map\mapsto
\newcommand\ot\leftarrow
\newcommand\lot\longleftarrow
\newcommand\lto\longrightarrow
\newcommand\tot\leftrightarrow
\newcommand\ltot\longleftrightarrow
\newcommand\sm\backslash
\renewcommand\Cup\bigcup
\renewcommand\Cap\bigcap
\newcommand\sub\subset
\newcommand\Sub\Subset
\newcommand\sne\subsetneq
\newcommand\bus\supset
\newcommand\Bus\Supset

\newcommand\eq\equiv
\newcommand\ox\otimes
\newcommand\Ox\bigotimes
\newcommand\pl\oplus
\newcommand\Pl\bigoplus
\newcommand\x\times
\renewcommand\c\circ
\newcommand\q\quad
\renewcommand\l\left
\renewcommand\r\right
\newcommand\fr\frac

\allowdisplaybreaks

\newcommand{\Lu}{\color{red}}

\def\sideremark#1{\ifvmode\leavevmode\fi\vadjust{\vbox to0pt{\vss% the remark
 \hbox to 0pt{\hskip\hsize\hskip1em%                          will appear only
 \vbox{\hsize2.1cm\tiny\raggedright\pretolerance10000%          on the side
  \noindent #1\hfill}\hss}\vbox to15pt{\vfil}\vss}}}%

\begin{document}

\title[Critical points of positive solutions]
{On the number and location of critical points of  solutions  of nonlinear elliptic equations in domains with a small hole}
\author[M. Grossi and P. Luo]{Massimo Grossi and  Peng Luo}

 \address[Massimo Grossi]{
Dipartimento di Matematica Guido Castelnuovo, Universit$\grave{a}$ Sapienza, P.le Aldo Moro 5, 00185 Roma, Italy}
 \email{grossi@mat.uniroma1.it}

 \address[Peng Luo]{School of Mathematics and Statistics and Hubei Key Laboratory of Mathematical Sciences, Central China Normal University, Wuhan 430079, China}
 \email{pluo@mail.ccnu.edu.cn}
\begin{abstract}
In this paper we study the following problem
\begin{equation*}
\begin{cases}
-\Delta u=f(u)~&\mbox{in}\ \Omega_\varepsilon,\\
u>0~&\mbox{in}\ \Omega_\varepsilon,\\
u=0~&\mbox{on}\ \partial\Omega_\varepsilon,
\end{cases}
\end{equation*}
where $\Omega_\varepsilon=\O\backslash B(P,\varepsilon)$, $\Omega\subset \R^N$ with $N\geq 2$ is a smooth bounded domain, $B(P,\varepsilon)$ is the ball centered at $P$ and radius $\varepsilon>0$ and $f$ is a smooth nonlinearity.

By some computations involving the Green  function and degree theory,
 we give the number and location of critical points of positive solutions for small $\varepsilon>0$.
\end{abstract}

\date{\today}
\maketitle

\section{Introduction and main results}\label{s0}
\setcounter{equation}{0}
A classic and fascinating problem of mathematical analysis  is  the number of the critical points of solutions  of the elliptic problem
 \begin{equation}\label{1}
\begin{cases}
-\Delta u=f(u)~&\mbox{in}\ \O,\\
u>0~&\mbox{in}\ \O,\\
u=0~&\mbox{on}\ \partial\O,
\end{cases}
\end{equation}
where $\Omega\subset \R^N$  is a smooth bounded domain,  $N\geq 2$ and $f$ is a smooth nonlinearity.

Problem \eqref{1} is a generalization of the elastic torsion problem, a classical topic
in PDEs, with references dating back to St. Venant(1856).
From then, many techniques and important results to address this problem were developed in the literature (Morse theory, degree theory, etc.). Despite the great interest aroused by the problem, many questions are still unanswered and we are far from a complete understanding of the phenomenon.

The calculation of the number of critical points of a function $u$ is  strictly related to the topological properties of the domain. This link is clearly highlighted in the following beautiful {\em Poincar\'e-Hopf Theorem} which we state in the particular case where $\O$ is a bounded smooth domain of $\R^N$. Here and in the rest of the paper $B(y,r)$ denote the ball centered at $y$ and radius $r$.
\vskip 0.1cm

\noindent \textbf{Theorem A} (Poincar\'e-Hopf Theorem). \emph{Let $\O\subset\R^N$, $N\ge2$, be a smooth bounded domain. Let $v$ be a vector field on $\O$ with isolated zeroes $x_1,..,x_k$ and such that $v(x)\cdot\nu(x)<0$ for any $x\in\partial\O$ (here $\nu$ is the outward normal vector to $\partial\O$). Then we have the formula
\begin{equation}\label{I2}
\sum_{i=1}^k index_{x_i}(v)=(-1)^N\chi(\O),
\end{equation}
where $index_x(v)=deg\big(v,B(x,\delta),0\big)$ with  small fixed $\delta>0$ and $\chi(\O)$ is the Euler characteristic of $\O$.
}

By $deg\big(v,B(x,\delta),0\big)$ we denote the classical Brower degree of a vector field $v$. Choosing $v=\nabla u$ in Theorem A we get a beautiful link between an {\em analytic }problem (to look for critical points of $u$) and a {\em topological }invariant (the Euler characteristic of $\O$). %Note that this theorem is more often stated when $v\cdot\nu>0$ on $\partial\O$ (the case where $v$ is pointing outward to $\partial\O$). However because we are considering $u>0$ in $\O$, by Hopf Lemma we have  $\nabla u\cdot\nu<0$ on $\partial\O$.

The first case studied in the literature is when $\O$ is a (strictly) convex domain. In this case $\chi(\O)=1$ and so \eqref{I2} becomes
\begin{equation}\label{I3}
\sum_{i=1}^k index_{x_i}(\nabla u)=(-1)^N.
\end{equation}
Of course since $u$ is a solution to \eqref{1}, we always have a maximum point for $u$ whose index is $(-1)^N$. The question is now
\vskip0.1cm\noindent
\begin{equation}\label{I4}
\hbox{\em when does the sum in \eqref{I3} reduce to a singleton?}
\end{equation}
Here we list some results that give an affirmative answer to the question \eqref{I4}.  Since it is impossible  to give an exhaustive bibliography we  limit ourselves to state some results that are closer to the interest of this paper.
\begin{itemize}
\item $f(s)=1$ and $\O\subset\R^2$ is a convex bounded domain (Makar-Limanov \cite{ML71}).
\item $f(s)=\lambda_1 s$ where $\lambda_1$ is the first eigenvalue of the Laplace operator and $\O\subset\R^N$ is a strictly convex bounded  domain (Acker,Payne and Philippin \cite{APP1981}, Brascamp and Lieb \cite{BL1976},  Korevaar  \cite{k}).
\item $f$ locally lipschitz and  $\O\subset\R^N$ is a symmetric bounded domain convex in any direction (Gidas, Ni and Nirenberg \cite{GNN1979}).
\item $f\ge0$, $\O\subset\R^2$ is a bounded domain with positive curvature and $u$ is a semi-stable solution to \eqref{1} (Cabr\'e and Chanillo \cite{CC1998}).
\end{itemize}
There are various conjectures on the uniqueness of the critical point of solutions of \eqref{1} under the mere assumption of the convexity of the domain but this seems to be a very difficult problem.

In this paper we consider the domain
$$\Omega_\varepsilon=\O\backslash B(P,\varepsilon)\hbox{ with $P\in\O$ and $\varepsilon$ small},$$
 and a solution $u_\e$ of
 \begin{equation}\label{aa2}
\begin{cases}
-\Delta u=f(u)~&\mbox{in}\ \O_\varepsilon,\\
u>0~&\mbox{in}\  \O_\varepsilon,\\
u=0~&\mbox{on}\ \partial\O_\varepsilon.
\end{cases}
\end{equation}
Next assumption on the solution $u_\e$ is crucial for our aim and it will be taken on throughout the paper,
\vskip0.2cm\noindent
\begin{equation}\label{2019-11-25-46}
|u_\varepsilon|\le C \hbox{ in }\O_\e\hbox{ with  $C$ independent of }\varepsilon.
\end{equation}
\vskip0.2cm\noindent
By the standard regularity theory, an immediate consequence of  \eqref{2019-11-25-46} is that there exist a sequence $\e_n\to0$ and $u_0\in C^2({\O})$, the solution to \eqref{1}, such that
\begin{equation}\label{I15}
\begin{cases}
\hbox{$u_{\e_n}\rightharpoonup u_0$ weakly in $H^1_0(\O)$ \big(here we extend $u_{\e_n}$ to $0$ in $B(P,\e)$\big),}\\
\hbox{$u_{\e_n}\to u_0$ in $C^2(K)$ for any compact set $K\subset\O\setminus P$.}
\end{cases}
\end{equation}
In all the paper we set $u_{\e_n}=u_\e$ and $u_0$ its weak limit.

If we write \eqref{I2} for $v=\nabla u_\e$ (again assuming that the number of critical point of $u_\e$ is finite) and denote by
\begin{equation*}
{\mathcal C}=\{\hbox{critical points of $u_0$ in $\O$}\}~\hbox{and}~\mathcal C_1=\{\nabla u_\e(x)=0\}\cap\{dist (x,\mathcal C)>\delta\},
\end{equation*}
and observing that
$\chi(\O_\e)=\chi(\O)+(-1)^{N-1}$, we get that if $P\not\in{\mathcal C}$ we have
%and observing that
%$\chi(\O_\e)=\chi(\O)+(-1)^{N-1}$ we get that \eqref{I3} becomes
%\begin{equation}\label{I5b}
%\sum_{x_i\in{\mathcal C}_1} index_{x_i}(\nabla u_\e)+\sum_{x_i \in{\mathcal C}_2} index_{x_i}(\nabla u_\e)=(-1)^N\chi(\O)-1,
%\end{equation}
%where $\mathcal C_1:=\{\nabla u_\e(x)=0\}\cap\{dist (x,\mathcal C)>\delta\}$
%and $\mathcal C_2:=\{\nabla u_\e(x)=0\}\cap\{dist (x,\mathcal C)\leq \delta\}$
%for any fixed small $\delta>0$. \\
%Now if $P\not\in{\mathcal C}$, using the convergence of $u_\e$ to $u_0$ we derive that $\sum_{x_i\in{\mathcal C_2}} index_{x_i}(\nabla u_\e)=\sum_{x_i\in{\mathcal C}} index_{x_i}(\nabla u)=
%(-1)^N\chi(\O)$ and so \eqref{I5b} becomes
\begin{equation}\label{I5}
\sum_{x_i \in{\mathcal C_1}} index_{x_i}(\nabla u_\e)=-1.
\end{equation}
Hence we have that the solution  $u_\e$ has at least one additional critical point which is away from $\mathcal C$ and  it must converge to $P$ as $\e\to0$.
As before a natural question arises<
\begin{equation}\label{I6}
\hbox{\em when does the sum in \eqref{I5} reduce to a singleton?}
\end{equation}
We will see that the answer to the question \eqref{I6} is positive in a quite general situation, as stated in the next theorem.
%Here and in all the paper we assume that  $\Omega$ is a smooth bounded domain of $\R^N$ with $N\geq 2$ and $f\in C^{1}$.
\begin{teo}\label{I7}
Suppose that $u_\e$ is a solution to \eqref{aa2} which  verifies \eqref{2019-11-25-46} and $u_0$ its weak limit. We have that if
\begin{equation}\label{I1}
P\hbox{ is not a critical point of }u_0,
\end{equation}
then for $\varepsilon$ small enough there is exactly one critical point for $u_\e$ in $B(P,d)\setminus B(P,\e)$
(here $B(P,d)\subset\O$ is chosen not containing any critical  point of $u_0$).
Moreover the  critical point $x_\e\in B(P,d)$ of $u_\e$ is a saddle point of index $-1$ which verifies
\begin{equation}\label{I16}
u_\e(x_\e)\to u_0(P),
\end{equation}
and
\begin{equation}\label{2019-11-27-03}
x_\e=P+
\begin{cases}
\Big(C_N+o(1)\Big)\nabla u_0(P)\e^\frac{N-2}{N-1}&~\mbox{for}~N\geq 3,\\
\Big(C_2+o(1)\Big)\nabla u_0(P)\frac{1}{|\log \varepsilon|}&~\mbox{for}~N=2,
\end{cases}
\end{equation}
where $C_N$ is given by
\begin{equation}\label{I23}
C_N=
\begin{cases}
-\left[\frac{(N-2)u_0(P)}{|\nabla u_0(P)|^N}\right]^\frac 1{N-1}&~\mbox{for}~N\geq 3,\\
-\frac{ u_0(P)}{ |\nabla u_0(P)|^2}&~\mbox{for}~N=2.
\end{cases}
\end{equation}
\end{teo}
\begin{rem}\label{I12}
The condition that $P$ is not a critical point of $u_0$ cannot be removed. An easy counterexample can be constructed when $\O=B(0,1)$ and $u_0$ is the first eigenfunction of $-\Delta$ with zero Dirichlet boundary condition. If $P=0$ we have that $\O_\e=B(0,1)\setminus B(0,\e)$ and $u_\e$ is the first $radial$ eigenfunction in the annulus $\O_\e$. Of course $u_\e$ has infinitely many critical points in $B(P,d)\setminus B(P,\e)$ for any $\e>0$ small and $d\in(0,1)$.
\end{rem}
\begin{rem}\label{rem1.3}
Let us give an idea of the proof of Theorem \ref{I7}.
The key step is to derive sharp $C^2$ expansions of the solution $u_\e$  which improve \eqref{I15}. For $N\ge3$ our basic estimate near $\partial B(P,\e)$ is the following
%(See Lemma \ref{lemma-1}, Propositions \ref{prop3} and \ref{prop2.4}),
\begin{equation}\label{aI13}
u_\varepsilon(x)=
u_0(x)-\frac{u_0(P)+o(1)}{|x-P|^{N-2}}\e^{N-2}+o(1).
\end{equation}
Note that near $\partial B(P,\e)$ there is an interaction between the weak limit $u_0$ and the fundamental solution of the Laplacian.

Another crucial result is to derive that $u_\e$ and $u_0(x)-\frac{u_0(P)}{|x-P|^{N-2}}\e^{N-2}$ are close in the $C^2$-topology  in $B(P,d)\setminus B(P,\e)$. Since the last function admits only one critical point which is also nondegenerate we get the uniqueness of the critical point for $u_\e$ in $B(P,d)\setminus B(P,\e)$.

Finally, again by \eqref{aI13} we get that the critical point of $u_\e(x)$ in $B(P,d)\setminus B(P,\e)$ can be founded as zero of the equation
$$0=\nabla u_0(x)+(N-2)\big(u_0(P)+o(1)\big)\frac{x-P}{|x-P|^N}\e^{N-2}+o(1),$$
which gives the formula \eqref{2019-11-27-03}. For $N=2$ similar computations occur.
%{\Lu Moreover let $$\Omega_{\e,2}:=\big\{C_1\e^{\frac{N-2}{N-1}}\leq  |x-P| \leq C_2\e^{\frac{N-2}{N-1}}~\mbox{with some}~C_1,C_2>0\big\},$$ we can find that(see Propositions \ref{prop3} and \ref{B35})
%\begin{equation*}
 %u_\e(x) =u_0(x)-\frac{u_0(P)+o(1)}{|x-P|^{N-2}}\e^{N-2}+o(1)~ \mbox{in}~ C^2\Big(\Omega_{\e,1}\bigcap \Omega_{\e,2}\Big).
%\end{equation*}
%This can be used to compute the Hessian matrix of $u_\e$ at the critical point $x_\e$.}
%
%The proof of  \eqref{aI13} and \eqref{I13} need several and delicate estimates which will be proved in Sections \ref{s3} and \ref{s4}.

\end{rem}

\begin{rem}
Even if $u_0$ has an isolated critical point $x_0$, we cannot exclude that there are critical points of $u_\e$ collapsing to $x_0$.
This will be rule out in next theorem assuming the nondegeneracy of the critical points of $u_0$.
\end{rem}
%A consequence of Theorem \ref{I7}
%is the following result, which allows to compute the number of critical points of $u_\e$.
\begin{teo}\label{th1-1}
Suppose that $u_\e$ is a solution to \eqref{aa2} which  verifies \eqref{2019-11-25-46}.
Denoting by $u_0$ its weak limit we get that if $P$ satisfies \eqref{I1} and all critical points of $u_0$ are nondegenerate we have that
\begin{equation}\label{mainresult}
\sharp\{\hbox{critical points of $u_\e$ in $\O_\e$}\}=\sharp\{\hbox{critical points of $u_0$ in $\O$}\}+1.
\end{equation}
Finally the additional critical point $x_\e$ of $u_\e$ is a saddle point of index $-1$ which verifies \eqref{2019-11-27-03}.
\end{teo}
Let us state some interesting situations where the previous theorem applies.

\begin{cor}\label{I11}
Assume that $\O$ is a symmetric domain with respect to the origin and convex in the directions $x_1,..,x_N$ and
suppose that $u_\e$ is a solution to \eqref{aa2} which  verifies \eqref{2019-11-25-46}.
Denoting by $u_0$ its weak limit we get that if $P\ne0$  we have that
\begin{equation*}
\sharp\{\hbox{critical points of $u_\e$ in $\O_\e$}\}=2,
\end{equation*}
and the additional critical point $x_\e$ of $u_\e$ is a saddle point of index $-1$ which verifies \eqref{2019-11-27-03}.
\end{cor}
\begin{proof}[Proof of Corollary \ref{I11}]
Since $u_0$ is a solution to \eqref{1} by the Gidas, Ni and Nirenberg Theorem we have that $0$ is  the unique  critical point to $u_0$. Moreover it is nondegenerate, as pointed out in \cite{G2002}. Then the claim follows by Theorem \ref{th1-1}.
\end{proof}
The previous corollary holds for the first eigenfunction of $ -\D$. If $\O_\e$ is a ball with a small hole we have a complete description of number of the critical points.
\begin{cor}
Assume that $\O_\e$ is the annular domain $B(0,1)\setminus B(P,\e)$ and $\phi_{1,\e}$  is the first eigenfunction of $ -\D$ in $B(0,1)\setminus B(P,\e)$. Then we have that
\begin{equation*}
\sharp\{\hbox{critical points of $\phi_{1,\e}$ in $\O_\e$}\}=
\begin{cases}
\infty~&\mbox{if}\ P=0,\\
2~&\mbox{if}\ P\ne0,
\end{cases}
\end{equation*}
for $\e$ small enough.
\end{cor}
Other examples where Theorem \ref{th1-1} applies  getting the existence of exactly $two$ critical points for the solution $u_\e$ in $\O_\e$ will be given in Section \ref{s6}.
\vskip 0.2cm
Next we consider the case $\nabla u_0(P)=0$. Here it is more complicated to prove results on the exact number of the critical points of $u_\e$.
As noted in Remark \ref{I12} it is even possible to have {\em infinitely many} critical points. Moreover, formula  \eqref{I5} becomes
 \begin{equation}\label{I20}
\sum_{x_i \in B(P,d)\setminus B(P,\e)} index_{x_i}(\nabla u_\e)=index_P(\nabla u_0)-1,
\end{equation}
where  $B(P,d)\subset\O$ is chosen not containing any critical  point of $u_0$. Hence the number of critical points of $u_\e$ in a neighborhood of $P$ is strongly depending of the index of $\nabla u_0$ at $P$. In particular, if $P$ is a maximum point for $u_0$ then \eqref{I20} becomes
  \begin{equation*}
\sum_{x_i \in B(P,d)\setminus B(P,\e)} index_{x_i}(\nabla u_\e)=(-1)^N-1,
\end{equation*}
%(note the difference with the case  $\nabla u_0(P)\ne0$).

 This case will be handled using analogous estimates to get asymptotic for $u_\e$ and its derivatives. Unfortunately some technical problems occur and for this we need an additional technical assumption.
\vskip 0.2cm
 \emph{Suppose that $u_\varepsilon$ and $u_0$ verify
\begin{equation}\label{2019-12-26-01}
\begin{split}
\int_{\Omega_\varepsilon}&\Big(f\big(u_\varepsilon(y)\big)-f\big(u_0(y)\big)\Big)
\frac{\partial G(x,y)}{\partial x_i}dy=
o\big(|x-P|\big)+
\begin{cases}
o\Big(\frac{\varepsilon^{N-2}}{|x-P|^{N-1}}\Big)&~\mbox{for}~N\geq 3,\\
o\Big( \frac{1}{|x-P| \cdot|\log \varepsilon|}\Big)&~\mbox{for}~N=2, \end{cases}
\end{split}
\end{equation}
when $|x-P|\rightarrow0$ and $\frac{|x-P|}{\varepsilon}\rightarrow +\infty$. Here $G(x,y)$ is the Green  function of $-\Delta$ in $\Omega$ with zero Dirichlet boundary condition.}

This assumption basically involves the rate of $\nabla(u_\e-u_0)$ near the critical point $P$. Some cases where it is verified are the following.
\begin{itemize}
\item  $f(s)\equiv1$.
\item $\O$ convex and symmetric with respect to $P$ as in the Gidas, Ni and Nirenberg Theorem.
(see Section \ref{s6}).
%\item $f(u)=\lambda_\varepsilon g(u)$, with $\lambda_\varepsilon=O(\e)$ and $g(u)\in C^1$.
\end{itemize}
%These cases will be discussed in Section \ref{s6}.
\vskip 0.2cm
Now let us state our main results.
Denote by $\textbf{H}(P)$ the Hessian matrix of $u_0$ at $P$ and suppose that $P$ is  a $nondegenerate$ critical point. Then
 denote by $m$ the number of its negative eigenvalues.  Our main result is the following.
\begin{teo}\label{th1-2}
Suppose that $u_\e$ is a solution to \eqref{aa2} which  verifies \eqref{2019-11-25-46}.
Denoting by $u_0$ its weak limit we get that if \eqref{2019-12-26-01} holds  and $P$ is a nondegenerate  critical point of $u_0$, for small $\varepsilon$ we have following results.
\vskip 0.2cm
\noindent \textup{(1)}  If  $\la_1<\la_2<..<\la_m<0$ are
the negative eigenvalues of $\textbf{H}(P)$ and they are simple we have that
\begin{equation}\label{adda2019-11-27-07}
\sharp\{\hbox{critical points of $u_\e$ in }\O_\e \}=\sharp\{\hbox{critical points of $u_0$  in }\O\}+2m-1.
\end{equation}
Moreover the additional critical points $x_{1,\e}^+,x_{1,\e}^-,..,x_{m,\e}^+,x_{m,\e}^-$ satisfy for $i=1,..,m$,
\begin{equation*}
u_\e(x_{i,\e}^\pm)\to u_0(P) \quad\hbox{for any }i=1,..,m
\end{equation*}
and
\begin{equation}\label{I21}
x_{j(i),\e}^\pm=P\pm
\begin{cases}
\left(\frac{(2-N)u_0(P)}{\la_i}+o(1)\right)^\frac1N\e^\frac{N-2}Nv_i&\hbox{if }N\ge3,\\
\sqrt{-\frac{u_0(P)}{\la_i}+o(1)}\frac1{\sqrt{|\log\e|}}v_i&\hbox{if }N=2,
\end{cases}
\end{equation}
with $v_i$ the $i$-th   eigenfunction associated to $\la_i$ and $j(1),..,j(m)$ is a permutation of indices $1,..,m$.
\vskip 0.2cm
\noindent \textup{(2)} If at least one negative eigenvalue of $\textbf{H}(P)$ is multiple, then
\begin{equation*}
\sharp\{\hbox{critical points of }u_\e\}\ge\sharp\{\hbox{critical points of }u_0\}+2m-1.
\end{equation*}
\end{teo}
\begin{rem}
The previous result tells us that for any direction outgoing from $P$ where $u_0$ is decreasing generates a pair of critical points. It is worth to note that if $P$ is a nondegenerate  minimum point for $u_0$(see Corollary \ref{C11}) then
\begin{equation*}
\sharp\{\hbox{critical points of }u_\e\}=\sharp\{\hbox{critical points of }u_0\}-1.
\end{equation*}
Finally radial solutions to \eqref{1} when $\O_\e$ is an annulus fall in case $(2)$ above.

\end{rem}
Next corollary computes the number of critical points of $u_\e$ for convex and symmetric domains.
\begin{cor}
Assume that $P=0$, $\O$ and $f$ are like in Gidas, Ni and Nirenberg Theorem. Then if all the eigenvalues of $H(P)$ are simple then for $\e$ small enough we have that
\begin{equation*}
\sharp\{\hbox{critical points of $u_\e$ in }\O_\e \}=2N
\end{equation*}
and $P_{1,\e}^+,P_{1,\e}^-,..,P_{N,\e}^+,P_{N,\e}^-$ satisfy \eqref{I21}.
\end{cor}
\begin{proof}
Since \eqref{2019-12-26-01} holds (see Section \ref{D3}) and by the Gidas, Ni and Nirenberg theorem $u_0$ has a unique critical point, the claim follows by \eqref{adda2019-11-27-07}.
\end{proof}
A consequence of the previous result is the location of maxima of radial solutions in annuli with shrinking hole.
\begin{cor}\label{cor1.7}
Let $\O_\e$ be the annulus $B(0,1)\setminus B(0,\e)$, $u_\e$ a radial solution to \eqref{aa2} and  $u_0$ its weak limit. Assume that $f(s)>0$  for $s>0$ and set $r=|x|$ and $u_\e=u_\e(r)$. We have that for $\e>0$ small enough $u_\e(r)$ has a unique critical point $r=r_\e$ given by
\begin{equation*}
r_\e=
 \begin{cases}
\left[\left(\frac{N(N-2)u_0(0)}{f\big(u_0(0)\big)}\right)^\frac1N+o(1)\right]\e^\frac{N-2}N&~\mbox{if}~N\geq 3,\\
\left(\sqrt{\frac{u_0(0)}{2f\big(u_0(0)\big)}}+o(1)\right)\frac1{\sqrt{|\log\e|}}&~\mbox{if}~N=2.
\end{cases}
\end{equation*}
\end{cor}
\vskip0.2cm
\begin{rem}
Both Theorems \ref{th1-1} and  \ref{th1-2} can be iterated to handle  the case in which $k$  small holes are removed from $\O$.
\end{rem}

The paper is organized as follows.  In Section \ref{s1} we recall some properties of the Green function. In Section \ref{s2} we split our solution $u_\e$ in different parts which will be estimated
in the next sections. Section  \ref{s3} contains some technical computations which allow  (in Section  \ref{s4}) to give the estimate of $u_\e$, $\nabla u_\e$ and $\nabla^2 u_\e$.
Section \ref{s5} is devoted to the proof of Theorems \ref{I7} and \ref{th1-1}.
In   Section \ref{s7} we give the main theorems when $\nabla u_0(P)=0$. Finally in Section \ref{s6} we give some applications and extensions of the previous theorems.

\vfill\eject

\section{Properties of the Green function}\label{s1}
In this section we collect some properties of the Green function which play a crucial role in the paper.
First we recall that, for $(x,y)\in\O\times\O$, $x\ne y$, the Green function $G(x,y)$ verifies
\begin{equation*}
\begin{cases}
-\D_x G(x,y)=\delta(y)&\hbox{in }\O,\\
G(x,y)=0&\hbox{on }\partial\O,
\end{cases}
\end{equation*}
in the sense of distribution. Next we recall the classical representation formula,
\begin{equation}\label{G}
G(x,y)=S(x,y)+H(x,y),
\end{equation}
where $S(x,y)$ is the classical {\em fundamental solution} given by
\begin{equation*}
S(x,y)=
\begin{cases}
-\frac{1}{2\pi}\log \big|x-y\big|&~\mbox{if}~N=2,\\
\frac{1}{N(N-2)\omega_N}\frac{1}{|x-y|^{N-2}} &~\mbox{if}~N\geq 3,
\end{cases}
\end{equation*}
with $\omega_N$ the volume of the unit ball $\R^N$ and $H(x,y)$ is the {\em regular part of the Green function}.\\
Since in the paper we need to consider the Green function in different domains, we denote by $G_U(x,y)$  as  the Green function on $U$. Next result is also classical.
\begin{teo}[Green's representation formula]
If $u\in C^2(\bar U)$, then it holds
\begin{equation}\label{grf}
u(x)=-\int_{\partial U}u(y)\frac{\partial G_U(x,y)}{\partial \nu_y}d\sigma(y)
-\int_{U} \Delta u(y) G_U(x,y)dy~\,~\mbox{for}~x\in U,
\end{equation}
where $\nu_y$ is the outer normal vector on $\partial U$.
\end{teo}
\begin{proof}
This can be found at page 19 of \cite{GT1983}.
\end{proof}
Let us denote by $G_0(w,z)$ the {\em Green function} of $\R^N\backslash B(0,1)$ given by (see \cite{BF1996})
\begin{equation}\label{new1}
G_0(w,z)=
\begin{cases}
-\frac{1}{2\pi}\left(\log \big|{w-z}\big|-\log \big||w|z-\frac{w}{|w|}\big|\right)&~\mbox{if}~N=2,\\
\frac{1}{N(N-2)\omega_N}\left(\frac{1}{|w-z|^{N-2}}-\frac{1}{\big||w|z-\frac{w}{|w|}\big|^{N-2}}\right) &~\mbox{if}~N\geq 3.
\end{cases}
\end{equation}
 We have the following computations.
\begin{lemma}\label{G1}
We have that,
 \begin{equation}\label{daaa2019-11-02-10}
\frac{\partial G_0(w,z)}{\partial \nu_z}=
\frac{1}{N\omega_N} \frac{1-|w|^2}{|w-z|^N},~\mbox{for}~|w|>1,~|z|=1~\mbox{and}~\nu_z=-z,
\end{equation}
\begin{equation}\label{2019-12-09-01}
\begin{split}
\frac{\partial G_0(w,z)}{\partial w_i}=O\Big(\frac{|z|}{|w-z|^{N-1}}\Big),~\mbox{for}~|w|,|z|>1,
\end{split}
\end{equation}
 \begin{equation}\label{2019-12-09-02}
\begin{split}
  \frac{\partial G_0(w,z)}{\partial w_i}
=&-\frac{1}{N\omega_N} \frac{w_i-z_i}{|w-z|^N} +O\Big(\frac{1}{|w|^{N-1}\cdot|z|^{N-2}}\Big),~\mbox{for}~|w|\to \infty,|z|>1.
\end{split}
\end{equation}
For any $\phi\in C^2\big(\overline{B(0,1)}\big)$ it holds
\begin{equation}\label{G2}
\phi(s)=\frac{1}{N\omega_N}\int_{\partial B(0,1)} \frac{1-|s|^2}{|s-y|^N}\phi(y)d\sigma(y)
-\int_{B(0,1)} \Delta \phi(y) G_0(s,y)dy.
\end{equation}
\end{lemma}
\begin{proof}
Formula \eqref{daaa2019-11-02-10} is a straightforward computation and  \eqref{G2} follows by \eqref{grf} and the well known Poisson kernel.
Concerning \eqref{2019-12-09-01} we have that
\begin{equation}\label{2019-12-08-29}
\frac{\partial G_0(w,z)}{\partial w_i} =
-\frac{1}{N\omega_N}\left(\frac{w_i-z_i}{|w-z|^N}-
\frac{w_i|z|^2-z_i}{\left(|w|^2|z|^2-2\langle w,z\rangle+1\right)^\frac N2}\right).
 \end{equation}
Since $\Big||z|w- \frac{z}{|z|}\Big|=\Big||w| z- \frac{w}{|w|}\Big|\geq |w-z|$ for $|w|\geq 1$ and $|z|\geq 1$, and $\Big|w_i|z|^2-z_i\Big|\le|z|\cdot\Big||z|w- \frac{z}{|z|}\Big|$, we find
\begin{equation*}
\begin{split}
\frac{\partial G_0(w,z)}{\partial w_i}
=&O\left(\frac{1}{|w-z|^{N-1}}+\frac{|z|}{\big||z|w- \frac{z}{|z|}\big|^{N-1}}\right)
=O\left(\frac{|z|}{|w-z|^{N-1}}\right).
\end{split}
\end{equation*}
Finally if
$|w|\to \infty$  we have $\left||z| w- \frac{z}{|z|}\right|\geq \frac{1}{2}|z|\cdot|w|$ and                                                                                                                                                                                                                                                                                                                                                                                                                                                            so by \eqref{2019-12-08-29} we get
 \begin{equation*}
\begin{split}
  \frac{\partial G_0(w,z)}{\partial w_i}
=&-\frac{1}{N\omega_N} \frac{w_i-z_i}{|w-z|^N} +O\left(\frac{|w|\cdot|z|^2+|z|}{\left||z|w- \frac{z}{|z|}\right|^N}\right),
\end{split}
\end{equation*}
which proves \eqref{2019-12-09-02}.
\end{proof}
\begin{rem}
Let us point out that  the {\em Green function} $G_0$ of $\R^N\backslash B(0,1)$  and
the {\em Poisson kernel} of $B(0,1)$  has the same formula (see \cite{BF1996}). This will be used to compute some  integral in $\R^N\backslash B(0,1)$.
\end{rem}
Now let us recall the {\em Newtonian potential} of a function $p\in C^{0,\a}(U)$ is given by
\begin{equation*}
L(x)=\int_U S(x,y)p(y)dy.
\end{equation*}
Next result computes the second derivative of the function $A$. It will be used in Section \ref{s4}.
\begin{lemma}\label{G5}
We have that
\begin{equation}\label{G33}
\frac{\partial^2  L(x)}{\partial x_i\partial x_j}=\int_{B_R(x)}\frac{\partial^2S(x,y)}{\partial x_i\partial x_j}\big(p(y)-p(x)\big)dy-\frac1Np(x)\delta_{ij},
\end{equation}
where $B_R(x)$ is any ball centered at $x$ containing $U$,  $\delta_{ij}$ is the Kronecker delta and $p$ is extended to vanish outside $U$.
\end{lemma}
\begin{proof}
It is a straightforward consequence of  Lemma 4.2 in \cite{GT1983}.
\end{proof}

Now we list some lemmas which will be used in the paper.
\begin{lemma}Let $u(x)$ is a $harmonic$ function in $\Omega$ and $B=B(x,R)\subset\subset \Omega$, then
\begin{equation}\label{a2019-11-02-01}
\big|\nabla u(x)\big|\leq \frac{N}{R}\sup_{\partial B(x,R)} |u|.
\end{equation}
\end{lemma}
\begin{proof}
This can be found at page 22 of \cite{GT1983}.
\end{proof}

Next lemma will be used in the proof of Lemma \ref{B3}.
 \begin{lemma}\label{G3}
Let $N=2$ and $\psi_\e$ be the function which verifies
 \begin{equation}\label{G4}
\begin{cases}
\Delta_x\psi_\e(x,z)=0&~\mbox{in}~\O\setminus B(P,\e),\\
\psi_\e(x,z)=1&~\mbox{on}~\partial\O,\\
\psi_\e(x,z)=0&~\mbox{on}~\partial B(P,\e).
\end{cases}
\end{equation}
Then we have that
\begin{equation*}
\psi_\e(x,z)=1+\frac{2\pi}{\log\e}\big( 1+o(1)\big)G(x,P)+O\left(\frac1{|\log\e|^2}\right).
\end{equation*}
\end{lemma}
\begin{proof}
 We set
$$\zeta_\e(x,z)=\frac1{H(P,P)}\left[\frac{\log\e}{2\pi}\big(\psi_\e(x,z)-1\big)-G(x,P)\right],$$
which solves
 \begin{equation*}
\begin{cases}
\Delta_x\zeta_\e(x,z)=0&~\mbox{in}~\O\setminus B(P,\e),\\
\zeta_\e(x,z)=0&~\mbox{on}~\partial\O,\\
\zeta_\e(x,z)=-\frac1{H(P,P)}\Big(\frac{\log\e}{2\pi}+G(x,P)\Big)=-1+O(\e)&~\mbox{on}~\partial B(P,\e).
\end{cases}
\end{equation*}
Hence repeating the same procedure we get
 \begin{equation*}
\begin{cases}
\Delta_x\left[\frac{\log\e}{2\pi}\zeta_\e(x,z)-G(x,P)\right]=0&~\mbox{in}~\O\setminus B(P,\e),\\
\frac{\log\e}{2\pi}\zeta_\e(x,z)-G(x,P)=0&~\mbox{on}~\partial\O,\\
\frac{\log\e}{2\pi}\zeta_\e(x,z)-G(x,P)=-H(x,P)+O(\e|\log\e|)&~\mbox{on}~\partial B(P,\e).
\end{cases}
\end{equation*}
  Then
by the maximum principle we get that
$$\frac{\log\e}{2\pi}\zeta_\e(x,z)-G(x,P)=O(1)~\,~\mbox{in}~\O\setminus B(P,\e),$$
which gives
\begin{equation*}
\begin{split}
\zeta_\e(x,z)=&\frac{2\pi}{\log\e}G(x,P)+O\left(\frac1{|\log\e|}\right).
\end{split}
\end{equation*}
Then we find
\begin{equation*}
\psi_\e(x,z)=1+\frac{2\pi}{\log\e}\Big( H(P,P)\zeta_\e(x,z)+G(x,P)\Big)=
1+\frac{2\pi}{\log\e}\big( 1+o(1)\big)G(x,P)+O\left(\frac1{|\log\e|^2}\right),
\end{equation*}
which gives the claim.
\end{proof}
\vfill\eject

\section{Splitting of the solution $u_\varepsilon$}\label{s2}
Let us write down the equation satisfied by $u_\e-u_0$ where $u_0$ and $u_\e$ are solutions of \eqref{1} and \eqref{aa2} respectively,
\begin{equation}\label{A1}
\begin{cases}
-\Delta \big(u_\varepsilon-u_0\big)=f\big(u_{\varepsilon}\big)-f\big(u_{0}\big)~&\mbox{in}~\Omega_\varepsilon,\\
u_\varepsilon-u_0=0~&\mbox{on}~\partial\Omega,\\
u_\varepsilon-u_0=-u_0~&\mbox{on}~\partial B(P,\varepsilon).
\end{cases}
\end{equation}
By Green's representation formula \eqref{grf},  we get
\begin{equation}\label{2019-11-25-01}
u_\varepsilon(x)=u_0(x)+
\int_{\partial B(P,\varepsilon)} \frac{\partial G_\varepsilon(x,y)}{\partial\nu_y}u_0(y)d\sigma(y)
+\int_{\Omega_\varepsilon} \Big(f\big(u_\varepsilon(y)\big)-f\big(u_0(y)\big)\Big)
G_\varepsilon(x,y)dy,
\end{equation}
where $\nu_y=-\frac{y-P}{|y-P|}$ is the outer normal vector of $\partial \big(\R^N\backslash  B(P,\varepsilon)\big)$ and $G_\varepsilon(x,y)$ is the Green function of $-\Delta$ in $\O_\varepsilon$ with zero Dirichlet boundary condition.
\vskip 0.2 cm
%Differentiating \eqref{2019-11-25-01} we get
%\begin{equation}\label{10}
%\begin{split}
% \frac{\partial u_\varepsilon(x)}{\partial x_i}=&\frac{\partial u_0(x)}{\partial x_i}+
%\int_{\partial B(P,\varepsilon)} \frac{\partial^2G_\varepsilon(x,y)}{\partial %x_i\partial\nu_y}u_0(y)d\sigma(y)+\frac{\partial }{\partial x_i}\left(\int_{\Omega_\varepsilon} %\Big(f\big(u_\varepsilon(y)\big)-f\big(u_0(y)\big)\Big)
%G_\varepsilon(x,y)dy\right).
%\end{split}
%\end{equation}

The behavior of $u_\e$ near $\partial B(P,\varepsilon)$ is crucial and a key point is to understand the limit of $G_\varepsilon(x,y)$ according to the location of $x$. We have the following two cases:
\vskip 0.2cm
\noindent\textup{(1)}  $x$ is far away from $P$, namely $|x-P|\ge C>0$.
\vskip 0.2cm
\noindent\textup{(2)}  $|x-P|=o(1)$.
\vskip 0.2cm
The first case is much easier as stated in the following lemma.
\begin{lemma}\label{lemma-1}
Let  $u_0$ and $u_\e$ be solutions of \eqref{1} and \eqref{aa2} respectively. Then for any fixed $R>0$, it holds
\begin{equation*}
 u_\varepsilon\rightarrow  u_0 ~\mbox{uniformly in}~C^2\big(\Omega\backslash B(P,R)\big).
\end{equation*}
\end{lemma}
\begin{proof}
Since $u_\e-u_0$ satisfies \eqref{A1} the claim follows by \eqref{2019-11-25-46} and the standard regularity theory.
\end{proof}
The rest of the paper is focused to estimate \eqref{2019-11-25-01} if $x$ is  approaching to $P$. It requires  delicate computations. We start by setting
$$x=P+\e w,\ y=P+\e z$$
and  introducing for $N\ge3$ the function $F_\e(w,z):\left(\frac{\O-P}\e\setminus B(0,1)\right)\times\left(\frac{\O-P}\e\setminus B(0,1)\right)$ defined as
\begin{equation}\label{2019-11-20-01}
F_\e(w,z)=
\e^{N-2}G_\varepsilon(P+\e w,P+\e z).
\end{equation}
A straightforward computation gives that
\begin{equation*}
\begin{cases}
-\D_w F_\varepsilon(w,z)=\delta(z)&\hbox{in }\frac{\O-P}\e\setminus B(0,1),\\
F_\varepsilon(w,z)=0&\hbox{on }\partial B(0,1).
\end{cases}
\end{equation*}
Note that $\frac{\O-P}\e\setminus B(0,1)\to \R^N\setminus B(0,1)$ and from the regularity theory we have
$F_\varepsilon(w,z)\to G_0(w,z) ~\hbox{in}~K$,
where $K$ is any compact set in $\R^N\setminus B(0,1)$.

Next, for $x=P+\e w$ we set
\begin{equation*}
I_{\e}(w)=\int_{\partial B(0,1)} \frac{\partial G_0(w,z)}{\partial\nu_z}u_0(P+\e z)d\sigma(z),
\end{equation*}
\begin{equation*}
J_{\e}(w)=\int_{\partial B(0,1)}\left(\frac{\partial F_\varepsilon(w,z)}{\partial\nu_z}-\frac{\partial G_0(w,z)}{\partial\nu_z}\right)u_0(P+\e z)d\sigma(z),
\end{equation*}
and
\begin{equation}\label{defA}
A_\varepsilon(x)=\int_{\Omega_\varepsilon} \Big(f\big(u_\varepsilon(y)\big)-f\big(u_0(y)\big)\Big)
G_\varepsilon(x,y)dy,
\end{equation}
where $\nu_z=-\frac{z}{|z|}$ is the outer normal vector of $\partial \big(\R^N\backslash  B(0,1)\big)$.
In this way, \eqref{2019-11-25-01}  becomes
\begin{equation}\label{m18}
\boxed{u_\varepsilon(x)=
u_0(x)+I_{\e}(w)+J_{\e}(w)+A_\varepsilon(x) }
\end{equation}
%\begin{equation}\label{2019-11-12-01}
%\boxed{\frac{\partial u_\varepsilon(x)}{\partial x_i}=
%\frac{\partial u_0(x)}{\partial x_i}+\frac{1}{\e}\frac{\partial I_{\e}(w)}{\partial w_i}+\frac{1}{\e}\frac{\partial J_{\e}(w)}{\partial w_i}+\frac{\partial A_\varepsilon(x)}{\partial x_i}}
%\end{equation}
which gives our fundamental splitting of $u_\e$. In next sections we estimate all terms of \eqref{m18} separately.
We will see that for $N\ge3$, the quantity $u_0(x)+I_{\e}(w)$ turns to be the leading term of the expansion. For $N=2$ the analogous of \eqref{2019-11-20-01} is not enough to prove our estimates. We have a more delicate situation which will described in Section \ref{s3}.
\vfill\eject

\section{Estimates for  $I_{\e}$, $J_{\e}$ and $A_\e$}\label{s3}
This section is divided in three parts where we estimate $I_{\e}$, $J_{\e}$ and $A_\varepsilon$ respectively.
\vskip 0.2cm
\subsection{Estimate of  $I_{\e}(w)=\displaystyle\int_{\partial B(0,1)} \frac{\partial G_0(w,z)}{\partial\nu_z}u_0(P+\e z)d\sigma(z)$}
\begin{lemma}
For $N\geq 2$ and $|w|>1$, it holds
\begin{equation}\label{a2019-11-15-05}
 I_{\e}(w)=-\frac{1}{|w|^{N-2}}\big(u_0(P)+o(1)\big),
 \end{equation}
\begin{equation}\label{ab2019-11-15-05}
\frac{\partial I_{\e}(w)}{\partial w_i}=\frac{(N-2)w_i}{|w|^N}\big(u_0(P)+o(1)\big)+
 O\left(\frac{\varepsilon}{|w|^N}\right),
 \end{equation}
and
\begin{equation}\label{2019-11-15-05b}
\begin{split}
\frac{\partial^2  I_{\e}(w)}{\partial w_i\partial w_j}=
 \frac{N-2}{|w|^N}\left(\delta_{ij}-\frac{Nw_iw_j}{|w|^2}+o\big(1\big)\right)
 + O\left(\frac{\varepsilon}{|w|^{N+1}}\right),
\end{split}
\end{equation}
where $\delta_{ij}$
is the Kronecker delta.
\end{lemma}
\begin{proof}
First, taking $s=\frac{w}{|w|^2}=\frac{\varepsilon(x-P)}{|x-P|^2}$ in \eqref{daaa2019-11-02-10} we get
\begin{equation*}
\begin{split}
\int_{\partial B(0,1)} &\frac{\partial G_0(w,z)}{\partial\nu_z}u_0(P+\e z)d\sigma(z)=-\frac{1}{N\omega_N|w|^{N-2}}\int_{\partial B(0,1)} \frac{1-|s|^2}{|s-z|^N}u_0(P+\e z)d\sigma(z).
\end{split}
\end{equation*}
For $|s|<1$ and  choosing $\phi(s)=u_0(P+\e s)$ in \eqref{G2} we find
\begin{equation*}
u_0(P+\e s)= \frac{1}{N\omega_N}\int_{\partial B(0,1)} \frac{1-|s|^2}{|s-z|^N}u_0(P+\e z)d\sigma(z)
+\varepsilon^2\int_{B(0,1)} f\big(u_0(P+\e z)\big)G_0(s,z)dz.
\end{equation*}
From the above computations  we  get
\begin{equation}\label{B22}
\begin{split}
I_{\e}(w)=&\int_{\partial B(0,1)}\frac{\partial G_0(w,z)}{\partial\nu_z}u_0(P+\e z)d\sigma(z)\\=&-\frac{u_0\left(P+ \frac{\varepsilon w}{|w|^2}\right)}
{|w|^{N-2}}+\frac{\varepsilon^2}{|w|^{N-2}}\underbrace{\int_{B(0,1)} f\big(u_0(P+\e z)\big)G_0(s,z)dz}_{=L(s)},
\end{split}
\end{equation}
and differentiating \eqref{B22} with respect to $w_i$
\begin{equation}\label{aaa2019-11-15-04}
\begin{split}
\frac{\partial I_{\e}(w)}{\partial w_i}=&
\frac{(N-2)w_i}{|w|^N}u_0\left(P+ \frac{\varepsilon w}{|w|^2}\right)
-\frac{\e}{|w|^{N}}\left(\frac{\partial u_0(P+ \frac{\varepsilon w}{|w|^2})}{\partial x_i}-2\frac{w_i}{ |w|^2}\sum_{j=1}^N\frac{\partial u_0(P+ \frac{\varepsilon w}{|w|^2})}{\partial x_j}w_j\right)\\&
-(N-2)\frac{\varepsilon^2w_i}{|w|^{N}}L\left(\frac w{|w|^2}\right)+\frac{\varepsilon^3}{|w|^{N}}
\left(\frac{\partial L}{\partial x_i}\left(\frac w{|w|^2}\right)+2\frac{w_i}{ |w|^2}\sum_{j=1}^Nw_j\frac{\partial L}{\partial x_j}\left(\frac w{|w|^2}\right)\right).
\end{split}
\end{equation}
Let us observe that $L$ verifies
\begin{equation}\label{B31}
\begin{cases}
-\Delta L= f\big(u_0(P+\e s)\big)&\hbox{in }B(0,1),\\
L=0&\hbox{on }\partial B(0,1),
\end{cases}
\end{equation}
and then $L(s)$, $\nabla L(s)$ and $\nabla^2 L(s)$ are uniformly bounded in $B(0,1)$. So we deduce \eqref{a2019-11-15-05} and  \eqref{ab2019-11-15-05}
by \eqref{B22} and  \eqref{aaa2019-11-15-04}.
Finally differentiating \eqref{aaa2019-11-15-04} with respect to $w_j$, we get \eqref{2019-11-15-05b} .
\end{proof}
 %\begin{rem}
%Note that if $N=2$ the leading term of $I_{i,\e}$ and $\frac{\partial  I_{i,\e}(w)}{\partial w_j}$ vanishes. A consequence is that $I_{i,\e}$ will not be the main term in \eqref{2019-11-12-01}.
%\end{rem}

\vskip 0.2cm

\subsection{Estimate of  $J_{\e}(w)=\displaystyle\int_{\partial B(0,1)}\left(\frac{\partial F_\varepsilon(w,z)}{\partial\nu_z}-\frac{\partial G_0(w,z)}{\partial\nu_z}\right)u_0(P+\e z)d\sigma(z)$}

\begin{lemma}\label{B3}
We have the following estimates.
\vskip 0.2cm\noindent
\begin{equation}\label{ad2019-11-21-23}
J_{\e}(w)=
\begin{cases}
O\left(\varepsilon^{N-2}\right) &\hbox{if }N\ge3,\\
\frac{\log |w|+2\pi H(P,P)}{|\log\e|}\big(u_0(P)+o(1)\big)&\hbox{if }N=2,
\end{cases}
\end{equation}
uniformly for $w\in\frac{\O-P}\e \setminus B(0,1)$.
\vskip 0.2cm\noindent
 If $|w|\rightarrow +\infty$ and $|x-P|=o(1)$, then it holds
\begin{equation}\label{2019-11-23-05}
\frac{\partial J_{\e}(w)}{\partial w_i}=
\begin{cases}
O\left(\frac{\varepsilon^{N-2}}{|w|}\right)~&\hbox{if }N\ge3,\\
\frac{w_i}{ |\log \e|\cdot|w|^2}\big(u_0(P)+o(1)\big) &\hbox{if }N=2,
\end{cases}
\end{equation}
and
\begin{equation}\label{B15}
\frac{\partial^2 {J}_{\e}(w)}{\partial w_i\partial  w_j}=
\begin{cases}
O\left(\frac{\varepsilon^{N-2}}{|w|^2}\right)~&\hbox{if }N\ge3,\\
\frac{ u_0(P)}{  |\log \varepsilon| \cdot |w|^2} \Big(\delta_{ij}-\frac{2w_iw_j}{|w|^2}+o\big(1\big)\Big)&\hbox{if }N=2.
\end{cases}
\end{equation}
\end{lemma}
\begin{proof}[
\underline{Proof of \eqref{ad2019-11-21-23}}.]
First consider $N\ge3$ and set
 \begin{equation*}
M_\e(w,z)=\sum_{i=1}^N\left(\frac{\partial G_0(w,z)}{\partial z_i}-\frac{\partial F_\varepsilon(w,z)}{\partial z_i}\right)z_i~\mbox{with}~|z|=1
\end{equation*}
and write down the equation satisfied by $M_\e$,
\begin{equation}\label{eqG}
\begin{cases}
\Delta_wM_\e=0&~\mbox{in}~\frac{\O-P}\e\setminus B(0,1),\\
M_\e=0&~\mbox{on}~\partial B(0,1),\\
M_\e=-\frac{\partial G_0(w,z)}{\partial\nu_z}=\frac{1}{N\omega_N} \frac{|w|^2-1}{|w-z|^N}&~\mbox{on}~\frac{\partial\O-P}\e.
\end{cases}
\end{equation}
Since $w=\frac{x-P}\e$, we have that if $x\in\partial\O$, then it holds
\begin{equation}\label{ddaeqG}
0<\frac{|w|^2-1}{|w-z|^N}=\frac{\left|\frac{x-P}\e\right|^2-1}{\left|\frac{x-P}\e-z\right|^N}=
\e^{N-2}\left(\frac1{|x-P|^{N-2}}+o(1)\right)\le\frac{2\e^{N-2}}{\big(\mbox{dist}(P,\partial\O)\big)^{N-2}}.
\end{equation}
Hence by the maximum principle for harmonic function, \eqref{eqG} and \eqref{ddaeqG}, we deduce that
\begin{equation}\label{est}
\sup_{w\in\frac{\O-P}\e \setminus B(0,1)}\left|M_\e\right|\le C\e^{N-2}.
\end{equation}
Recalling that
\begin{equation*}
J_{\e}(w)=\int_{\partial B(0,1)}\left(\frac{\partial F_\varepsilon(w,z)}{\partial\nu_z}-\frac{\partial G_0(w,z)}{\partial\nu_z}\right)u_0(P+\e z)d\sigma(z),
\end{equation*}
we have that \eqref{ad2019-11-21-23} follows for $N\geq 3$.

If $N=2$ the proof of \eqref{ad2019-11-21-23} requires better approximations. So we introduce the following function:
$$M_{\e,2}(w,z)=\sum_{i=1}^2\left(\frac{\partial G_0(w,z)}{\partial z_i}-\frac{\partial F_\varepsilon(w,z)}{\partial z_i}\right)z_i-\frac{1}{2\pi}\widetilde M_\e(w,z)$$
with
\begin{equation*}
\begin{cases}
\Delta_w\widetilde{M}_\e(w,z)=0&~\mbox{in}~\frac{\O-P}\e\setminus B(0,1),\\
\widetilde{M}_\e(w,z)=0&~\mbox{on}~\partial B(0,1),\\
\widetilde{M}_\e(w,z)=1&~\mbox{on}~\frac{\partial\O-P}\e.
\end{cases}
\end{equation*}
Hence $J_{\e}(w)$ can be written as
\begin{equation}\label{B5}
\begin{split}
J_{\e}(w)=&\underbrace{\frac{1}{2\pi}\int_{\partial B(0,1)} \widetilde{M}_{\e}(w,z)u_0(P+\e z)d\sigma(z)}_{:=\widetilde{J}_{\e,1}(w)} +\underbrace{\int_{\partial B(0,1)} M_{\e,2}(w,z)u_0(P+\e z)d\sigma(z)}_{:=\widetilde{J}_{\e,2}(w)}.
\end{split}
\end{equation}
Next, we find
\begin{equation*}
\begin{cases}
\Delta_wM_{\e,2}(w,z)=0&~\mbox{in}~\frac{\O-P}\e\setminus B(0,1),\\
M_{\e,2}(w,z)=0&~\mbox{on}~\partial B(0,1),\\
M_{\e,2}(w,z)=\frac{1}{2\pi}\left(\frac{|w|^2-1}{|w-z|^2}-1\right)&~\mbox{on}~\frac{\partial\O-P}\e.
\end{cases}
\end{equation*}
Since for any $w\in \frac{\partial\O-P}\e$,  we get that
\begin{equation*}
\begin{split}
 \frac{|w|^2-1}{|w-z|^2}-1=& O\Big(\frac{1}{|w|}\Big)=O\Big(\e\Big).
\end{split}
\end{equation*}
Then  by the maximum principle, we find
$$|M_{\e,2}(w,z)|=O(\varepsilon)~\mbox{for}~w\in\frac{\O-P}\e \setminus B(0,1),$$
and then \begin{equation}\label{m4}
\widetilde{J}_{\e,2}(w)=O(\varepsilon)~\mbox{for}~w\in\frac{\O-P}\e \setminus B(0,1).
\end{equation}

Next we estimate $\widetilde{M}_\e(w,z)$. To do this let us introduce the function $\psi_\e(x,z):\O\to\R$ as $\psi_\e(x,z)=\widetilde M_\e\left(\frac{x-P}\e,z\right)$ which solves \eqref{G4}. Then by Lemma \ref{G3} we have that
\begin{equation*}
\psi_\e(x,z)=1+\frac{2\pi}{\log\e}\big( 1+o(1)\big)G(x,P)+O\left(\frac1{|\log\e|^2}\right).
\end{equation*}
 Coming back to $\widetilde M_\e$, we get
 \begin{equation*}
\begin{split}
 \widetilde M_\e(w,z)=&1+\frac{2\pi}{\log\e}\big( 1+o(1)\big)G(\varepsilon w+P,P)+O\left(\frac1{|\log\e|^2}\right)\\=&
\frac{\log |w|+2\pi H(P,P)}{|\log\e|}
+o\left(\frac{1}{|\log\e|}\right).
\end{split}
\end{equation*}
 In last estimate  we used that $\varepsilon |w|=|x-P|\rightarrow 0$. And then
 \begin{equation}\label{m3}
\begin{split}
\widetilde{J}_{\e,1}(w)= \frac{\log |w|+2\pi H(P,P)}{|\log\e|}\big(u_0(P)+o(1)\big)
+o\left(\frac{1}{|\log\e|}\right).
\end{split}
\end{equation}
Then \eqref{ad2019-11-21-23}
follows by \eqref{B5}, \eqref{m4} and \eqref{m3} when $N=2$.
\vskip 0.2cm\noindent

\noindent{\em \underline{Proof of \eqref{2019-11-23-05} and \eqref{B15}.}}
 Now we remark that, since $B\big(w,\frac{|w|-1}{2}\big) \subset\subset \frac{\O-P}\e\setminus B(0,1)$, using  \eqref{a2019-11-02-01} and \eqref{est}, we get
\begin{equation}\label{2019-11-12-02}
\begin{split}
&\big|\nabla_w M_\varepsilon(w,z) \big| \le \frac{2N}{|w|-1}\sup_{\zeta\in\partial B\big(w,\frac{|w|-1}{2}\big)}\left|M_\varepsilon(\zeta,z)\right|=
O\left(\frac{\e^{N-2}}{|w|-1}\right),~\mbox{for}~N\geq 3.
\end{split}\end{equation}
So if $|w|\rightarrow +\infty$ and $|x-P|=o(1)$, we have
\begin{equation*}
\begin{split}
&\Big|\frac{\partial J_{\e}(w)}{\partial w_i}\Big|\le \int_{\partial B(0,1)}\left|\nabla_w M_\varepsilon(w,z) \right| u_0(P+\e z)d\sigma(z)=O\left(\frac{\e^{N-2}}{|w|}\right),~\mbox{for}~N\geq 3,
\end{split}\end{equation*}
which gives \eqref{2019-11-23-05}.
Moreover since $\frac{\partial  M_\varepsilon(w,z)}{\partial w_i}$ is again a harmonic function, using again \eqref{2019-11-12-02} we find
\begin{equation*}
\begin{split}
&\left|\nabla_w^2 M_\varepsilon(w,z) \right| \le \frac{4N^2}{(|w|-1)^2}\sup_{\zeta\in\partial B\big(w,\frac{|w|-1}{2}\big)}\left| M_\varepsilon(\zeta,z) \right|=
O\left(\frac{\e^{N-2}}{|w|^2}\right),
\end{split}\end{equation*}
which implies that $\frac{\partial^2 J_{\e}(w)}{\partial w_i\partial w_j}=O\left(\frac{\e^{N-2}}{|w|^2}\right)$. This proves \eqref{B15} and it ends the proof for $N\ge3$.

Similarly, if $N=2$ and  $|w|\to \infty$, from \eqref{a2019-11-02-01} and \eqref{m3}, it holds
\begin{equation}\label{2019-11-21-33}
\frac{\partial \widetilde{J}_{\e,2}(w)}{\partial w_i}=O\left(\frac{\e}{|w|}\right)~\mbox{and}~
\frac{\partial^2 \widetilde{J}_{\e,2}(w)}{\partial w_i\partial w_j}=O\left(\frac\e{|w|^2}\right).
\end{equation}
 Also $\widetilde M_\e(w,z)-
\frac{\log |w|}{|\log\e|}$ is a harmonic function with respect to $w$ in $\frac{\O-P}\e\setminus B(0,1)$, then using \eqref{a2019-11-02-01} we have
\begin{equation}\label{2019-11-24-18}
\left|\nabla_w\left(\widetilde M_\e(w,z)-
\frac{\log |w|}{|\log\e|}\right)\right|=o\left(\frac{1}{(|w|-1)|\log\e|}\right)
\end{equation}
and
\begin{equation}\label{2019-11-24-19}
\left|\nabla_w^2\left(\widetilde M_\e(w,z)-
\frac{\log |w|}{|\log\e|}\right)\right|=o\left(\frac{1}{(|w|-1)^2|\log\e|}\right).
\end{equation}
Hence from \eqref{2019-11-24-18}, we have
\begin{equation}\label{2019-11-21-35}
\frac{\partial \widetilde{J}_{\e,1}(w)}{\partial w_i}
=\frac{w_i}{ |\log \varepsilon|\cdot |w|^2}\big(u_0(P)+o(1)\big),
\end{equation}
and \eqref{2019-11-24-19} gives
\begin{equation}\label{2019-11-21-36}
\frac{\partial^2 \widetilde{J}_{\e,1}(w)}{\partial w_i\partial w_j}
=\frac{u_0(P)}{ |\log \varepsilon|\cdot |w|^2}\left(\delta_{ij}-\frac{2w_iw_j}{|w|^2}+o\big(1\big)\right).
\end{equation}
Hence \eqref{2019-11-23-05} and \eqref{B15} can be deduced by \eqref{B5},
\eqref{2019-11-21-33},  \eqref{2019-11-21-35} and \eqref{2019-11-21-36} for $N=2$.
\end{proof}
\vskip 0.2cm
\subsection{Estimates of  $A_\varepsilon(x)=\displaystyle\int_{\Omega_\varepsilon} \Big(f\big(u_\varepsilon(y)\big)-f\big(u_0(y)\big)\Big)
G_\varepsilon(x,y)dy$}~\\
In this section we estimate $A_\varepsilon(x)$ and $\frac{\partial A_\varepsilon(x)}{\partial x_i}$. The second derivative of $A_\e$ will be considered in the next section. Set
$p_\e=f(u_\e)-f(u_0)$.
We have the following result.

\begin{lemma}\label{adL1}
We have that
\begin{equation}\label{B11}
 A_\varepsilon\to0\hbox{ uniformly in }\O.
\end{equation}
If $|x-P|\varepsilon^{-1}\rightarrow\infty$  and $|x-P|\rightarrow 0$,
it holds
\begin{equation}\label{aa2019-11-25-10}
\frac{\partial A_\varepsilon(x)}{\partial x_i}=
\begin{cases}
 o\left(1+\frac{\e^{N-2}}{|x-P|^{N-1}}\right)&\hbox{if }N\ge3,\\
 o\Big(1+\frac{1}{|x-P|\cdot |\log \varepsilon|}\Big)&\hbox{if }N=2.
\end{cases}
\end{equation}
\end{lemma}
\begin{proof}[\underline{Proof of \eqref{B11}}]  Let us  introduce  $\widetilde{A_\e}$ which verifies
\begin{equation*}
\begin{cases}
-\Delta\widetilde{A_\e}=\widetilde{p}_\e&~\mbox{in}~\O,\\
 \widetilde{A_\e}=0&~\mbox{on}~\partial\O,
\end{cases}
\end{equation*}
where $\widetilde{p}_\e$ is the extension to zero of $p_\e$ to $0$ in $\O$, namely
\begin{equation*}
\widetilde{p}_\e(x)=\begin{cases}
p_\e(x)&~\mbox{if}~x\in\O_\e,\\
0&~\mbox{if}~x\in B(P,\e).
\end{cases}
\end{equation*}
Since $\widetilde{p}_\e\to0$ a.e in $\O$ and by \eqref{2019-11-25-46}, we find that $|\widetilde{p}_\e|\le C$ with $C$ independent of $\e$. And then by the standard regularity theory we get that $\widetilde{A_\e}\to0$ $uniformly$ in $\O$. On the other hand we have that $\widetilde{A_\e}-A_\e$ is a harmonic function in $\O_\e$ and then
\begin{equation*}
\inf\limits_{x\in\partial\O_\e}\widetilde{A_\e}(x)\le\widetilde{A_\e}(x)-
A_\e(x)\le\sup\limits_{x\in\partial\O_\e}\widetilde{A_\e}(x).
\end{equation*}
Hence  $\widetilde{A_\e}-A_\e\to0$ uniformly in $\O$ and this implies \eqref{B11}.

\vskip 0.2cm

\noindent\underline{\emph{Proof of \eqref{aa2019-11-25-10}.}}
Setting $x=P+\e w$ and $y=P+\e z$ we get
$$\frac{\partial G_\e(x,y)}{\partial x_i}
=\frac1{\e^{N-1}}\frac{\partial G_0\left(w,z\right)}{\partial w_i}
+\frac1{\e^{N-1}}\Big(\frac{\partial F_\varepsilon\left(w,z\right)}{\partial w_i}
-\frac{\partial G_0\left(w,z\right)}{\partial w_i}\Big),$$
and from \eqref{defA} we have
\begin{equation}\label{2019-12-10-16}
\begin{split}
\frac{\partial  A_\varepsilon(x)}{\partial x_i}= &
 \underbrace{\frac1{\e^{N-1}}\int_{\Omega_\varepsilon}
\Big(f\big(u_\varepsilon(y)\big)-f\big(u_0(y)\big)\Big) \frac{\partial
G_0\left(w,\frac{y-P}\e\right)}{\partial w_i}dy}_{:=K_{i,1}(w)}  \\&
+  \underbrace{\frac1{\e^{N-1}}\int_{\Omega_\varepsilon}
\Big(f\big(u_\varepsilon(y)\big)-f\big(u_0(y)\big)\Big)
\left(\frac{\partial F_\varepsilon\left(w,\frac{y-P}\e\right)}{\partial w_i}
-\frac{\partial G_0\left(w,\frac{y-P}\e\right)}{\partial w_i}\right)dy}_{:=K_{i,2}(w)}.
 \end{split}\end{equation}
Let $N\ge3$,  we start again by the decomposition in \eqref{2019-12-10-16}. If
$|w|\to \infty$  we have                                                                                                                                                                                                                                                                                                                                                                                                                                                              from  \eqref{2019-12-09-02} and the dominate convergence theorem that
\begin{equation}\label{2019-12-10-33}
\begin{split}
K_{i,1}(w)= & -\frac{1}{N\omega_N} \int_{\Omega_\varepsilon}\Big(f\big(u_\varepsilon(y)\big)-f\big(u_0(y)\big)\Big)
\frac{x_i-y_i}{|x-y|^{N}}dy \\&
+O\Big(\frac{\varepsilon^{N-2}}{|x-P|^{N-1}}\Big)
\int_{\Omega_\varepsilon}\Big|f\big(u_\varepsilon(y)\big)-f\big(u_0(y)\big)\Big|
\frac{1}{|y-P|^{N-2}}dy\\=&
  -\frac{1}{N\omega_N} \int_{\Omega_\varepsilon}\Big(f\big(u_\varepsilon(y)\big)-f\big(u_0(y)\big)\Big)
\frac{x_i-y_i}{|x-y|^{N}}dy +o\left(\frac{\varepsilon^{N-2}}{|x-P|^{N-1}}\right).
\end{split}
\end{equation}
Now let
\begin{equation}\label{2019-12-10-23}
\gamma_\varepsilon(w,z)=\frac{ \partial F_\varepsilon(w,z)}{\partial w_i}-\frac{ \partial G_0(w,z)}{\partial w_i}+\widetilde{\gamma}_\varepsilon(w,z)
\end{equation}
with
\begin{equation}\label{2019-12-10-25}
\begin{cases}
\Delta_z \widetilde{\gamma}_\varepsilon(w,z) =0&~\mbox{for}~z\in \frac{\O-P}\e\setminus B(0,1),\\
\widetilde{\gamma}_\varepsilon(w,z)=0&~\mbox{for}~z\in\partial B(0,1),\\
\widetilde{\gamma}_\varepsilon(w,z)=-
\frac{1}{N\omega_N} \frac{w_i-z_i}{|w-z|^N}&~\mbox{for}~z\in\frac{\partial\O-P}\e.
\end{cases}
\end{equation}
Then we find
\begin{equation*}
\begin{cases}
\Delta_z {\gamma}_\varepsilon(w,z) =0&~\mbox{for}~z\in \frac{\O-P}\e\setminus B(0,1),\\
 {\gamma}_\varepsilon(w,z)=0&~\mbox{for}~z\in\partial B(0,1),\\
 {\gamma}_\varepsilon(w,z)=-
\frac{1}{N\omega_N}  \frac{w_i|z|^2-z_i}{\left(|w|^2|z|^2-2\langle w,z\rangle+1\right)^\frac
N2}&~\mbox{for}~z\in\frac{\partial\O-P}\e.
\end{cases}
\end{equation*}
Hence, for $z\in\frac{\partial\O-P}\e$ and arguing similarly to the proof of Lemma \ref{G1}, it holds
\begin{equation*}
\left| \frac{w_i|z|^2-z_i}{\left(|w|^2|z|^2-2\langle w,z\rangle+1\right)^\frac
N2}\right|\leq \frac{C|z|}{(|w|\cdot |z|)^{N-1}}=\frac{C\varepsilon^{2N-3}}{|x-P|^{N-1}}.
\end{equation*}
Then by the maximum principle we get
\begin{equation}\label{2019-12-10-22}
\sup_{z\in\frac{\O-P}\e\setminus B(0,1)}|\gamma_\varepsilon(w,z)|\leq  \frac{C\varepsilon^{2N-3}}{|x-P|^{N-1}}.
\end{equation}
Hence from \eqref{2019-12-10-16},  \eqref{2019-12-10-23}, \eqref{2019-12-10-22} and the dominate convergence theorem, we have
\begin{equation}\label{2019-12-10-32}
\begin{split}
 K_{i,2}(w)=\frac1{\e^{N-1}}\int_{\Omega_\varepsilon}
\Big(f\big(u_\varepsilon(y)\big)-f\big(u_0(y)\big)\Big)
\widetilde{\gamma}_\e\left(w,\frac{y-P}{\varepsilon}\right)dy +o\Big(\frac{\varepsilon^{N-2}}{|x-P|^{N-1}}\Big).
 \end{split}\end{equation}
 Next we estimate $\widetilde{\gamma}_\varepsilon(w,z)$. Let us   define
 $\varphi_\e(x,y):\O\times\O\to\R$ as $\varphi_\e(x,y)=\widetilde\gamma_\e\left(\frac{x-P}{\varepsilon},\frac{y-P}\e\right)$ and from \eqref{2019-12-10-25} it verifies
 \begin{equation*}
\begin{cases}
\Delta_y\varphi_\e(x,y)=0&~ \mbox{for}~y\in \O\setminus B(P,\e),\\
\varphi_\e(x,y)=0&~\mbox{for}~y\in\partial B(P,\e),\\
\varphi_\e(x,y)=-
\frac{1}{N\omega_N} \frac{x_i-y_i}{|x-y|^N}\varepsilon^{N-1}&~\mbox{for}~y\in \partial\O.
\end{cases}
\end{equation*}
Recalling the decomposition of the Green function in \eqref{G}, we set
$$\xi_\e(x,y)=\varphi_\e(x,y)+\frac{\partial H(x,y)}{\partial x_i}\varepsilon^{N-1}
-\frac{\partial H(x,P)}{\partial x_i}\frac{\varepsilon^{2N-3}}{|y-P|^{N-2}},$$
%Here we know $\frac{\partial S(x,y)}{\partial x_i}=-\frac{1}{N\omega_N}\frac{x_i-y_i}{|y-x|^N}$.
which satisfies
 \begin{equation*}
\begin{cases}
\Delta_y \xi_\e(x,y)=0&~\mbox{for}~y\in \O\setminus B(P,\e),\\
\xi_\e(x,y)=\varepsilon^{N-1}\Big (\frac{\partial H(x,y)}{\partial x_i}
-\frac{\partial H(x,P)}{\partial x_i}\Big)=O\big(\varepsilon^{N}\big)&~\mbox{for}~y\in\partial B(P,\e),\\
\xi_\e(x,y)=-\frac{\partial H(x,P)}{\partial x_i}\frac{\varepsilon^{2N-3}}{|y-P|^{N-2}}=
O\big(\varepsilon^{2N-3}\big)&~\mbox{for}~y\in \partial\O.
\end{cases}
\end{equation*}
Hence
by the maximum principle we get
\begin{equation*}
\begin{split} \xi_\e(x,y)=O\Big(\varepsilon^{N}\Big)\,\,\,~\mbox{for}~y\in\O\setminus B(P,\e),
\end{split}\end{equation*}
and then
\begin{equation*}
\begin{split} \varphi_\e(x,y)=-\frac{\partial H(x,y)}{\partial x_i}\varepsilon^{N-1}
+ O\left(\frac{\varepsilon^{2N-3}}{|y-P|^{N-2}} +\varepsilon^{N}\right)\,\,\,~\mbox{in}~\O\setminus B(P,\e).
\end{split}\end{equation*}
Coming back to the initial function $\widetilde{\gamma}_\e(w,z)$, we get for $x=P+\e w$,
\begin{equation}\label{a2019-12-10-32}
\begin{split}
\widetilde{\gamma}_\e(w,z)
\Big|_{z=\frac{y-P}{\varepsilon}}=&-\frac{\partial H(x,y)}{\partial x_i}\varepsilon^{N-1}
+O\left(\frac{\varepsilon^{2N-3}}{|y-P|^{N-2}}+\varepsilon^{N}\right)\,\,\,~\mbox{for}~x\in\O\setminus B(P,\e).
\end{split}\end{equation}
Hence from \eqref{2019-12-10-32} and \eqref{a2019-12-10-32},
  we find
\begin{equation*}
\begin{split}
K_{i,2}(w)=
- \int_{\Omega_\varepsilon}\Big(f\big(u_\varepsilon(y)\big)-f\big(u_0(y)\big)\Big)
\frac{\partial H(x,y)}{\partial x_i}dy+o\left(\frac{\varepsilon^{N-2}}{|x-P|^{N-1}} +\varepsilon\right),
\end{split}\end{equation*}
which jointly with \eqref{2019-12-10-33} and \eqref{2019-11-25-46} imply
\begin{equation}\label{2019-12-26-02}
\begin{split}
\frac{\partial  A_\varepsilon(x)}{\partial x_i}=&   \int_{\Omega_\varepsilon}\Big(f\big(u_\varepsilon(y)\big)-f\big(u_0(y)\big)\Big)
\frac{\partial G(x,y)}{\partial x_i}dy+o\left(\frac{\varepsilon^{N-2}}{|x-P|^{N-1}}+\varepsilon\right)\\=
&o\left(1+\frac{\varepsilon^{N-2}}{|x-P|^{N-1}}\right).
 \end{split}\end{equation}
So  \eqref{aa2019-11-25-10} follows for $N\geq 3$.

Lastly we consider the case $N=2$. First we write \eqref{2019-12-08-29} in the following way,
  \begin{equation}\label{B12}
  \begin{split}
  \frac{\partial G_0(w,z)}{\partial w_i} =&
  -\frac{1}{2\pi}\left(\frac{w_i-z_i}{|w-z|^2}-\frac{w_i|z|^2-z_i}{|w|^2|z|^2-2w\cdot
  z+1}\right)\\=&
-\frac{1}{2\pi}\left(\frac{w_i-z_i}{|w-z|^2}- \frac{w_i}{|w|^2}\right)
 + \underbrace{\frac{1}{2\pi}\frac{2\langle w,z\rangle w_i-z_i|w|^2-w_i}{
 \left(|w|^2|z|^2-2\langle w,z\rangle+1\right)|w|^2}}_{:=\widetilde{K}_{i,1}(w,z)}.
  \end{split}\end{equation}
Since $|z|\geq 1$ and $|w|\to \infty$, we find
$$\big(|w|^2|z|^2-2\langle w,z\rangle+1\big)^{\frac12}\geq
|w|\cdot|z|-1\geq \frac{C-1}{C}|w|\cdot|z|,$$
and
$$\Big|2\langle w,z\rangle w_i-z_i|w|^2-w_i\Big|\leq
4|w|^2\cdot|z|.$$
Then it holds
\begin{equation}\label{2019-12-10-02}
\widetilde{K}_{i,1}(w,z)=
O\left(\frac{1}{|w|^2\cdot|z|}\right)=
O\left(\frac{\varepsilon^3}{|x-P|^2\cdot|y-P|}\right).
  \end{equation}
Hence from \eqref{2019-12-10-16}, \eqref{B12} and \eqref{2019-12-10-02},  we get
\begin{equation}\label{2019-12-13-05}
\begin{split}
K_{i,1}(w)
=&\int_{\Omega_\varepsilon}
      \Big(f\big(u_\varepsilon(y)\big)-f\big(u_0(y)\big)\Big)
      \Big(\frac{\partial S(x,y)}{\partial x_i}-\frac{\partial S(x,P)}{\partial x_i}\Big)
      dy
 +o\left(\frac{\varepsilon^2}{|x-P|^2}\right).
\end{split}
\end{equation}
Now we compute the term $K_{i,2}(w)$ in \eqref{2019-12-10-16}. Analogously to \eqref{2019-12-10-23} set
\begin{equation}\label{2019-12-10-62}
\delta_\e(w,z)=\frac{\partial F_\varepsilon(w,z)}{\partial w_i}
-\frac{\partial G_0(w,z)}{\partial w_i}-\frac{w_i\log |z|}{2\pi|w|^2\log\varepsilon}-\widetilde{\delta}_\e(w,z),\end{equation}
where $\widetilde{\delta}_\e$  verifies
\begin{equation*}
\begin{cases}
\Delta_z \widetilde{\delta}_\e(w,z) =0&~\mbox{for}~z\in \frac{\O-P}\e\setminus B(0,1),\\
\widetilde{\delta}_\e(w,z)=0&~\mbox{for}~z\in\partial B(0,1),\\
\widetilde{\delta}_\e(w,z)=
 \frac{1}{2\pi} \frac{w_i-z_i}{|w-z|^2}&~\mbox{for}~z\in\frac{\partial\O-P}\e.
\end{cases}
\end{equation*}
So we get  for $\delta_\e$
\begin{equation*}
\begin{cases}
\Delta_z {\delta}_\e(w,z) =0&~\mbox{for}~z\in \frac{\O-P}\e\setminus B(0,1),\\
 {\delta}_\e(w,z)=0&~\mbox{for}~z\in\partial B(0,1),\\
 {\delta}_\e(w,z)=
\frac{1}{2\pi}  \frac{w_i|z|^2-z_i}{\left(|w|^2|z|^2-2\langle w,z\rangle+1\right)}+\frac{w_i\log |z|}{2\pi|w|^2\log\varepsilon}&~\mbox{for}~z\in\frac{\partial\O-P}\e.
\end{cases}
\end{equation*}
Observe that for $z\in\frac{\partial\O-P}\e$, arguing as in \eqref{2019-12-10-02} we have (here $\zeta=\e z+P\in\partial\O)$
\begin{equation*}
\begin{split}
\frac{1}{2\pi} & \frac{w_i|z|^2-z_i}{\left(|w|^2|z|^2-2\langle w,z\rangle+1\right)}+\frac{w_i\log |z|}{2\pi|w|^2\log\varepsilon}\\&=
\frac{1}{2\pi}\left[ \frac{w_i|z|^2-z_i}{\left(|w|^2|z|^2-2\langle w,z\rangle+1\right)}-\frac{w_i}{|w|^2}
\right]+\frac{w_i\log |\zeta-P|}{2\pi|w|^2\log\varepsilon}\\&=
O\left(\frac{\varepsilon^3}{|x-P|^2}\right)+O\left(\frac{\varepsilon}{|x-P|\cdot|\log\varepsilon|}\right)
=O\left(\frac{\varepsilon}{|x-P|\cdot|\log\varepsilon|}\right).
\end{split}\end{equation*}
Then the maximum principle gives us that
\begin{equation}\label{2019-12-10-61}
\sup_{z\in\frac{\partial\O-P}\e}{\delta}_\e(w,z)
\leq  \frac{C\varepsilon}{|x-P|\cdot|\log\varepsilon|}.
\end{equation}
Hence by \eqref{2019-12-10-16}, \eqref{2019-12-10-62} and \eqref{2019-12-10-61}, we have
\begin{equation}\label{2019-12-13-01}
\begin{split}
K_{i,2}(w)=& \frac1{\e}\int_{\Omega_\varepsilon}
\Big(f\big(u_\varepsilon(y)\big)-f\big(u_0(y)\big)\Big)
\widetilde{\delta}_\e\left(w,\frac{y-P}{\varepsilon}\right)dy
 \\&
+\frac1{\e}\frac{w_i}{|w|^2}\int_{\Omega_\varepsilon}
\Big(f\big(u_\varepsilon(y)\big)-f\big(u_0(y)\big)\Big)
\frac{\log\left|\frac{y-P}{\varepsilon}\right|}{2\pi\log\varepsilon}dy
+o\left(\frac{1}{|x-P|\cdot|\log\varepsilon|}\right)\\=&
\frac1{\e}\int_{\Omega_\varepsilon}
\Big(f\big(u_\varepsilon(y)\big)-f\big(u_0(y)\big)\Big)
\widetilde{\delta}_\e\left(w,\frac{y-P}{\varepsilon}\right)dy
+o\left(\frac{1}{|x-P|\cdot|\log\varepsilon|}\right)\\&
+\frac{\partial S(x,P)}{\partial x_i}\int_{\Omega_\varepsilon}
      \Big(f\big(u_\varepsilon(y)\big)-f\big(u_0(y)\big)\Big) dy.
 \end{split}\end{equation}
Finally we estimate the term $\widetilde{\delta}_\e(w,z)$. First let us   define  $\phi_\e(x,y):\O_\e\times\O_\e\to\R$ as $\phi_\e(x,y)=\widetilde\delta_\e\left(\frac{x-P}{\e},\frac{y-P}\e\right)$, which verifies
 \begin{equation*}
\begin{cases}
\Delta_y\phi_\e(x,y)=0&~ \mbox{for}~y\in \O\setminus B(P,\e),\\
\phi_\e(x,y)=0&~\mbox{for}~y\in\partial B(P,\e),\\
\phi_\e(x,y)= -\frac{\partial S(x,y)}{\partial x_i}\e&~\mbox{for}~y\in \partial\O.
\end{cases}
\end{equation*}
Then we set
$$\widetilde{\phi}_\e(x,y)= \phi_\e(x,y)-\frac{\partial H(x,y)}{\partial x_i}\e
+\frac{\partial H(x,P)}{\partial x_i}\frac{\varepsilon }{\log \varepsilon}\log |y-P|,$$
which satisfies
 \begin{equation*}
\begin{cases}
\Delta_y \widetilde{\phi}_\e(x,y)=0&~\mbox{for}~y\in \O\setminus B(P,\e),\\
\widetilde{\phi}_\e(w,y)= \left(-\frac{\partial H(x,y)}{\partial x_i}+\frac{\partial H(x,P)}{\partial x_i}\right)\e=O\big(\e^2\big)&~\mbox{for}~y\in\partial B(P,\e),\\
\widetilde{\phi}_\e(x,y)=\frac{\partial H_\varepsilon(x,P)}{\partial x_i}\frac{\varepsilon}{\log \varepsilon}\log |y-P|=O\Big(\frac{\varepsilon }{|\log\e|}\Big)&~\mbox{for}~y\in \partial\O.
\end{cases}
\end{equation*}
Then
by the maximum principle we get that
 \begin{equation}\label{2019-12-13-02}
 \widetilde{\delta}_\e(w,z)=\phi_\e(x,y)=- \frac{\partial H(x,y)}{\partial x_i}\e+O\Big(
\frac{\varepsilon}{|\log \varepsilon| }\Big)\,\,\,~\mbox{in}~\O\setminus B(P,\e).
\end{equation}
Hence from \eqref{2019-12-13-01} and \eqref{2019-12-13-02},
  we find
\begin{equation}\label{2019-12-08-03}
\begin{split}
K_{i,2}(w)=
\int_{\Omega_\varepsilon}
      \Big(f\big(u_\varepsilon(y)\big)-f\big(u_0(y)\big)\Big)
      \Big(\frac{\partial H(x,y)}{\partial x_i}+\frac{\partial S(x,P)}{\partial x_i}\Big)
      dy
+o\left(\frac{1}{|x-P|\cdot|\log\varepsilon|}\right).
\end{split}\end{equation}
Then \eqref{2019-12-13-05} and \eqref{2019-12-08-03} imply
\begin{equation}\label{2019-12-26-02ad}
\begin{split}
\frac{\partial  A_\varepsilon(x)}{\partial x_i}=&   \int_{\Omega_\varepsilon}\Big(f\big(u_\varepsilon(y)\big)-f\big(u_0(y)\big)\Big)
\frac{\partial G(x,y)}{\partial x_i}dy+o\left(\frac{1}{|x-P|\cdot|\log\varepsilon|}\right)\\=
&o(1)+o\left(\frac{1}{|x-P|\cdot|\log\varepsilon|}\right),
 \end{split}\end{equation}
 which proves \eqref{aa2019-11-25-10} for $N=2$ and ends the proof.
 \end{proof}
\vfill\eject

\section{Estimates for $u_\e$, $\nabla u_\e$ and $\nabla^2u_\e$}\label{s4}
In this section we  write some expansions for $u_\e$, $\nabla u_\e$ and $\nabla^2u_\e$ in $\O_\e$.
A first consequence of the estimate of the previous section is the following result which extends Lemma \ref{lemma-1}.
\begin{prop}\label{B16}
If  $|x-P|=o(1)$ and $\frac{|x-P|}{\varepsilon}\rightarrow +\infty$, we have that
\begin{equation}\label{ll2}
u_\e(x)-u_0(x)\to0~\mbox{when}~N\geq 3.
\end{equation}
Moreover, if $\frac{\log |x-P|}{\log \e}\to 0$, then it holds
\begin{equation}\label{2020-02-03-4}
u_\e(x)-u_0(x)\to0 ~\mbox{when}~N=2.
\end{equation}
\end{prop}
\begin{proof}
From \eqref{2019-11-25-01} we have
\begin{equation*}
\begin{split}
u_\e(x)-u_0(x)=&
\int_{\partial B(P,\varepsilon)} \frac{\partial G_\varepsilon(x,y)}{\partial\nu_y}u_0(y)d\sigma(y)
+ A_\varepsilon(x)\\=
& \e^{N-2}\int_{\partial B(0,1)} \frac{\partial  G_\varepsilon(P+\e w,P+\e z)}{ \partial\nu_z}u_0(P+\e z)d\sigma(z) +  A_\varepsilon(x) \\=
& \int_{\partial B(0,1)} \frac{\partial F_\varepsilon(w,z)}{ \partial\nu_z}u_0(P+\e z)d\sigma(z) +  A_\varepsilon(x).
\end{split}
\end{equation*}
Now if $N\ge3$, arguing as in  \eqref{ddaeqG} and \eqref{est}, if  $\frac{|x-P|}{\varepsilon}\rightarrow +\infty$  we get
\begin{equation}\label{B8}
\begin{split}
 \int_{\partial B(0,1)} \frac{\partial F_\varepsilon(w,z)}{ \partial\nu_z}u_0(P+\e z)d\sigma(z)
=O\left(\frac{\e^{N-2}}{|x-P|^{N-2}}\right)=o\big(1\big).
\end{split}
\end{equation}
So the claim \eqref{ll2} follows by \eqref{B11} and \eqref{B8}.

If $N=2$, from \eqref{B22}, \eqref{B5}, \eqref{m4} and \eqref{m3}, we find
\begin{equation}
\begin{split}
 \int_{\partial B(0,1)} &\frac{\partial F_\varepsilon(w,z)}{ \partial\nu_z}u_0(P+\e z)d\sigma(z)
\\=
&-u_0\left(P+ \frac{\varepsilon w}{|w|^2}\right)+\varepsilon^2\underbrace{\int_{B(0,1)} f\big(u_0(P+\e z)\big)G_0(\frac{w}{|w|^2},z)dz}_{=O(1)\hbox{ by }\eqref{B31}
} \\&-\frac{\log|w|}{2\pi\log \e }\int_{\partial B(0,1)} u_0(P+\e z)d\sigma(z)+
O\Big(\frac{1}{|\log \e|}\Big)\\=&
-u_0(P)+o(1)-\frac{\log|x-P|-\log\e}{\log \e }\Big( u_0(P)+o(1)\Big)
\\=&\frac{\log|x-P|}{\log \e }\Big( u_0(P)+o(1)\Big)+o(1) =o(1),
\end{split}\end{equation}
by assumption. This and \eqref{B11} give \eqref{2020-02-03-4}.

\end{proof}
Next two propositions state some fundamental estimates for  $u_\e$.
\begin{prop}\label{prop3}
 Let  $u_0$ and $u_\varepsilon$ be the solutions to \eqref{1} and \eqref{aa2} respectively. Then for any fixed $R>1$,
\vskip 0.2cm\noindent
\begin{equation}\label{m10}
u_\e(x)=
\begin{cases}
u_0(x)-\frac{u_0(P)}{|x-P|^{N-2}}\e^{N-2}+o(1)&\hbox{if }N\ge3,\\
u_0(x)-\frac{\log |x-P|+2\pi H(P,P)}{\log\e}u_0(P)+o(1)&\hbox{if }N=2,
\end{cases}
\end{equation}
\vskip 0.2cm\noindent
in $C^2\big(B(P,R\e)\setminus B(P,\e)\big).$
\end{prop}
\begin{proof}
Set, for $x=P+\e w$,
\begin{equation*}
\eta_\e(w)=u_\e(P+\e w)-u_0(P+\e w)
\end{equation*}
which verifies the equation
\begin{equation*}
-\D\eta_\e=\e^2\Big[f\big(u_\e(P+\e w)\big)-f\big(u_0(P+\e w)\big)\Big]=\e^2c_\e(w)\eta_\e\quad\hbox{in }\frac{\O_\e-P}\e,
\end{equation*}
with $$c_\e(w)=\int_0^1f'\big(tu_\e(P+\e w)+(1-t)u_0(P+\e w)\big)dt.$$
Since $\eta_\e$ and $c_\e$ are bounded in $L^\infty$ by the standard regularity theory we get that
\begin{equation}\label{m1}
\eta_\e\to \eta_0\quad\hbox{in }C^2\Big(B(0,R)\setminus B(0,1)\Big).
\end{equation}
On the other hand, using the decomposition \eqref{m18},
 \eqref{a2019-11-15-05}, \eqref{ad2019-11-21-23} and \eqref{B11} we deduce
\begin{equation*}
\begin{split}
\eta_\e(w)=&u_\e(P+\e w)-u_0(P+\e w)=I_{\e}(w)+J_{\e}(w)+A_\varepsilon(x)\\=
&\begin{cases}
-\frac{u_0(P)+o(1)}{|w|^{N-2}}+o(1)&\hbox{if }N\ge3,\\
-u_0(P)+\frac{\log |w|+2\pi H(P,P)}{|\log\e|}\big(u_0(P)+o(1)\big)+o(1)&\hbox{if }N=2,
\end{cases}
\end{split}
\end{equation*}
which implies that $\eta_0(w)=\begin{cases}
-\frac{u_0(P)}{|w|^{N-2}}&\hbox{if }N\ge3,\\
-u_0(P)+\frac{\log |w|+2\pi H(P,P)}{|\log\e|}u_0(P)&\hbox{if }N=2.
\end{cases}$

So by \eqref{m1}, we find \eqref{m10}.

\end{proof}
\begin{prop}\label{prop2.4}
Let  $u_0$ and $u_\varepsilon$ be the solutions to \eqref{1} and \eqref{aa2} respectively. Then we have that,
\begin{equation}\label{2019-11-15-01}
u_\e(x)=
\begin{cases}
u_0(x)-\frac{u_0(P)+o(1)}{|x-P|^{N-2}}\e^{N-2}+o(1)&\hbox{if }N\ge3,\\
u_0(x)-\big(u_0(P)+o(1)\big)\frac{\log |x-P|}{\log\e}+o(1)&\hbox{if }N=2,
\end{cases}
\end{equation}
in $C^1(D_\e)$ with $D_\e=\left\{x\in\Omega_\e,~\frac{|x-P|}{\varepsilon}\rightarrow +\infty\hbox{ and }|x-P|=o(1)\right\}$.
\end{prop}
\begin{proof}
We only prove \eqref{2019-11-15-01} for $\nabla u_\e$. We observe that by  \eqref{m18} we have that
\begin{equation*}
\frac{\partial u_\e(x)}{\partial x_i}=\frac{\partial u_0(x)}{\partial x_i}+\frac1\e\frac{\partial I_{\e}(w)}{\partial w_i}+\frac1\e\frac{\partial J_{\e}(w)}{\partial w_i}+\frac1\e\frac{\partial A_\e(x)}{\partial x_i}.
\end{equation*}

Let $N\ge3$,  $|w|=\frac{|x-P|}{\varepsilon}\rightarrow +\infty$ and $|x-P|=o(1)$.
%$$\frac{\partial J_{\e}(w)}{\partial w_i}=O\left(\frac{\varepsilon^{N-2}}{|w|}\right)=o(1)$$
Hence using again the decomposition \eqref{m18}, \eqref{a2019-11-15-05}, \eqref{ab2019-11-15-05},  \eqref{2019-11-23-05} and \eqref{aa2019-11-25-10},  we have
\begin{equation*}
\begin{split}
\frac{\partial u_\e(x)}{\partial x_i}=&\frac{\partial u_0(x)}{\partial x_i}
+\big(u_0(P)+o(1)\big)\frac{(N-2)w_i}{\varepsilon|w|^N}+
 o(1)+o\left(1+\frac{\e^{N-2}}{|x-P|^{N-1}}\right)\\=
&\frac{\partial u_0(x)}{\partial x_i}+o(1)+(N-2)\big(u_0(P)+o(1)\big)\frac{x_i-P_i}{|x-P|^N}\e^{N-2}.
\end{split}\end{equation*}
So the claim \eqref{2019-11-15-01} follows for $N\geq 3$.

In the same way, if $N=2$ we get
\begin{equation*}
\begin{split}
\frac{\partial u_\e(x)}{\partial x_i}=&\frac{\partial u_0(x)}{\partial x_i}
-\frac{w_i}{\e\log \e\cdot|w|^2}\big(u_0(P)+o(1)\big)+o\left(1+\frac{1}{|x-P|\cdot |\log \varepsilon|}\right)\\=
&\frac{\partial u_0(x)}{\partial x_i}
+o(1)-\big(u_0(P)+o(1)\big)\frac{x_i-P_i}{\log\e\cdot|x-P|^2}\,
\end{split}\end{equation*}
which gives the claim \eqref{2019-11-15-01} follows for $N=2$.
\end{proof}

A first interesting consequence of the previous estimate is the following $necessary$ condition on the location of critical points ``close'' to $\partial B(P,\e)$.
\begin{prop}\label{B10}
If $x_\e$ is a critical point of $u_\e$ such that $x_\e\to P$ as $\e\to0$ and $\nabla u_0(P)\neq 0$, then we have that
\begin{equation}\label{2019-11-27-04}
x_\e=P+
\begin{cases}
\Big(C_N+o(1)\Big)\nabla u_0(P)\e^\frac{N-2}{N-1}~&\mbox{for}~N\geq 3,\\
\Big(C_2+o(1)\Big)\nabla u_0(P) \frac{1}{|\log \varepsilon|}~&\mbox{for}~N=2,
\end{cases}
\end{equation}
where
$C_N=-\left[\frac{(N-2)u_0(P)}{|\nabla u_0(P)|^N}\right]^\frac 1{N-1}$ for $N\geq 3$ and $C_2=-\frac{ u_0(P)}{ |\nabla u_0(P)|^2}$.
\end{prop}
\begin{proof}
By Propositions \ref{prop3} and \ref{prop2.4}, if $\nabla u(x_\e)=0$  we have that
\begin{equation}\label{B14}
-\frac{\partial u_0(P)}{\partial x_i}
=o(1)+\begin{cases}
 (N-2)\big(u_0(P)+o(1)\big)\frac{x_i-P_i}{|x-P|^N}\e^{N-2}&~\mbox{if}~N\geq 3,\\
\frac{x_i-P_i}{|\log\e|\cdot|x-P|^2}\big(u_0(P)+o(1)\big)&~\mbox{if}~N=2,
 \end{cases}
\end{equation}
and then  for $N\ge3$ we find
\begin{equation*}
\lim\limits_{\e\to0}\frac{\e^{N-2}}{|x_\e-P|^{N-1}}=\frac{|\nabla u_0(P)|}{(N-2)u_0(P)}.
\end{equation*}
Jointly with  \eqref{B14} this gives the claim \eqref{2019-11-27-04}. The case $N=2$ is analogous.
\end{proof}
Next aim is to estimate the second derivative of $A_\e$. This case is more complicated than the previous ones and also uses Propositions \ref{B16} and \ref{B10}. Since the proofs are quite long we separate the cases $N\ge3$ and $N=2$.
\begin{lemma}\label{B17}
Assume $N\geq 3$ and $
C_1\e^{\beta}\leq  |x-P| \leq C_2\e^{\beta}$
for some $C_1,C_2>0$ and $\beta\in (0,\frac{N-1}{N})$. Then it holds
\begin{equation}\label{2019-11-25-35}
\frac{\partial^2 A_\varepsilon(x)}{\partial x_i\partial x_j}=o\left(1 +\frac{\varepsilon^{N-2}}{|x-P|^{N}}\right).
\end{equation}
\end{lemma}
\begin{rem}\label{re}
According to Proposition \ref{B10} we have that critical points $x_\e$ which converge to $P$ necessarily satisfy $C_1\e^{\frac{N-2}{N-1}}\leq  |x_\e-P| \leq C_2\e^{\frac{N-2}{N-1}}$. This is the case $\beta=\frac{N-2}{N-1}$ in Lemma \ref{B17}.  Since we will be interested to compute the second derivatives at $x_\e$ this assumption is not restrictive.
\end{rem}
\begin{proof}[Proof of Lemma \ref{B17}]
Setting $U=\O_\e$ in Lemma \ref{G5} and
$p_\e=f(u_\e)-f(u_0)$,
we get by \eqref{G33},
\begin{equation}\label{B7}
\begin{split}
&\frac{\partial^2  A_\varepsilon(x)}{\partial x_i\partial x_j}=\int_{B_R(x)}\frac{\partial^2S(x,y)}{\partial x_i\partial x_j}\big(p_\e(y)-p_\e(x)\big)dy-\frac1Np_\e(x)\delta_{ij}+\int_{\O_\e}\frac{\partial^2  H_\varepsilon(x,y)}{\partial x_i\partial x_j}p_\e(y)dy.
 \end{split}\end{equation}
Analogously to  \eqref{2019-11-20-01} let us introduce
\begin{equation*}
E_\e(w,z)=\e^{N-2}H_\varepsilon(P+\e w,P+\e z),
\end{equation*}
and rewrite \eqref{B7} as follows
\begin{equation}\label{2019-12-10-41}
\begin{split}
\frac{\partial^2  A_\varepsilon(x)}{\partial x_i\partial x_j}= &
\underbrace{\int_{B_R(x)}\frac{\partial^2S(x,y)}{\partial x_i\partial x_j}\big(p_\e(y)-p_\e(x)\big)dy-\frac1Np_\e(x)\delta_{ij}}_{=\overline{K}_0(w)} \\
&+\underbrace{\frac1{\e^{N}}\int_{\Omega_\varepsilon}
 \frac{\partial^2H_0\left(w,\frac{y-P}\e\right)}{\partial w_i\partial w_j}p_\e(y)
dy}_{=\overline{K}_1(w)}  \\&
+\underbrace{ \frac1{\e^{N}}\int_{\Omega_\varepsilon}
\left(\frac{\partial^2 E_\varepsilon\left(w,\frac{y-P}\e\right)}{\partial w_i\partial w_j}
-\frac{\partial^2 H_0\left(w,\frac{y-P}\e\right)}{\partial w_i\partial w_j}\right)p_\e(y)
dy}_{=\overline{K}_2(w)},
 \end{split}\end{equation}
where $H_0(w,z)$ is the regular part of $G_0(w,z)$ in \eqref{new1}.

Let us estimate $\overline{K}_0(w)$, $\overline{K}_1(w)$, $\overline{K}_2(w)$.
\vskip0.2cm

\noindent \underline{Estimate of $\overline{K}_0(w)$}
\vskip0.2cm

First we note that by Proposition \ref{B16} we have that
$p_\e(x)=o(1)$.
Now we estimate the integral in  $\overline{K}_0(w)$.
\begin{equation*}
\begin{split}
\int_{\O_\e} &\frac{\partial^2  S(x,y)}{\partial x_i\partial x_j}\big(p_\e(y)-p_\e(x)\big)dy
\\=& \int_{\O_\e\backslash \widetilde{\O}_\e} \frac{\partial^2  S(x,y)}{\partial x_i\partial x_j}\big(p_\e(y)-p_\e(x)\big)dy
+\int_{\widetilde{\O}_\e} \frac{\partial^2  S(x,y)}{\partial x_i\partial x_j}\big(p_\e(y)-p_\e(x)\big)dy,
\end{split}\end{equation*}
where $\widetilde{\O}_\e=B\big(P,|x-P|^{\frac{N-\a}{N-1}}\big)\backslash B(P,\e)$ with  $\a\in (0,1)$
and ${\frac{\beta(N-\a)}{N-1}}<1$.  It is immediate to verify that in this case $\widetilde{\O}_\e$ is nonempty.
Since $|x-y|\ge|x-P|-\e$, we find
\begin{equation*}
\begin{split}|x-y|^N\ge|x-P|^N\left(1-\frac\e{|x-P|}\right)^N=
|x-P|^N\left(1+O(\e^\frac{1-\a}{N-\a})\right)^N\ge\frac12|x-P|^N.
\end{split}\end{equation*}
Also we have that $\frac{\partial^2  S(x,y)}{\partial x_i\partial x_j}=O\left(\frac{1}{|x-P|^N}\right)$ for $y\in \widetilde{\O}_\e$. Moreover, being
 $p_\e$ bounded in $\Omega_\e$, we find
\begin{equation*}
 \int_{\widetilde{\O}_\e} \frac{\partial^2  S(x,y)}{\partial x_i\partial x_j}\big(p_\e(y)-p_\e(x)\big)dy
 =O\left(|x-P|^{\frac{N(1-\a)}{N-1}}\right)=o(1),
 \end{equation*}
 because $|x-P|\to0$.
Next,
\begin{equation}\label{B19}
\begin{split}
\int_{\O_\e\backslash \widetilde{\O}_\e} &\frac{\partial^2  S(x,y)}{\partial x_i\partial x_j}\big(p_\e(y)-p_\e(x)\big)dy
\\=&-\int_{\O_\e\backslash \widetilde{\O}_\e} \frac{\partial^2  S(x,y)}{\partial x_i\partial y_j}\big(p_\e(y)-p_\e(x)\big)dy\\=&
-\int_{\partial \big(\O_\e\backslash \widetilde{\O}_\e\big)} \frac{\partial  S(x,y)}{\partial x_i }\big(p_\e(y)-p_\e(x)\big)\nu_jd\sigma(y)
+\int_{\O_\e\backslash \widetilde{\O}_\e} \frac{\partial  S(x,y)}{\partial x_i }\frac{\partial p_\e(y)}{\partial y_j}dy,
\end{split}\end{equation}
and using that $u_\e,u_0\big|_{\partial\O}=0$ we get
\begin{equation*}
\begin{split}
&\int_{\partial \big(\O_\e\backslash \widetilde{\O}_\e\big)} \frac{\partial  S(x,y)}{\partial x_i }\big(p_\e(y)-p_\e(x)\big)\nu_jd\sigma(y)\\=&
-p_\e(x)\int_{\partial \Omega} \frac{\partial  S(x,y)}{\partial x_i }\nu_jd\sigma(y)-\int_{\partial B\big(P,|x-P|^{\frac{N-\alpha}{N-1}}\big)}\underbrace{ \frac{\partial S(x,y)}{\partial x_i}}_{=O\left(|x-P|^{N-1}\right)}\big(p_\e(y)-p_\e(x)\big)\nu_jd\sigma(y)\\=&
O\Big(|p_\e(x)|\Big) +O\Big( |x-P|^\frac{(2-\alpha)N-1}{N-1} \Big)=o(1).
\end{split}\end{equation*}
By Proposition \ref{B16}, the last integral is estimated as follows,
\begin{equation}\label{B20}
\begin{split}
\int_{\O_\e\backslash \widetilde{\O}_\e} &\frac{\partial S(x,y)}{\partial x_i}\frac{\partial p_\e(y)}{\partial y_j}dy\\=&
\int_{\O_\e\backslash \widetilde{\O}_\e} \frac{\partial S(x,y)}{\partial x_i }\Big(
f'\big(u_\e(y)\big)\frac{ \partial u_\e(y)}{\partial y_j}-f'\big(u_0(y)\big)\frac{ \partial u_0(y)}{\partial y_j}\Big)dy\\=&
\int_{\O_\e\backslash \widetilde{\O}_\e} \frac{\partial S(x,y)}{\partial x_i }\Big(
f'\big(u_\e(y)\big) -f'\big(u_0(y)\big)\Big)\frac{ \partial u_0(y)}{\partial y_j}dy\\+&
\int_{\O_\e\backslash \widetilde{\O}_\e} \frac{\partial S(x,y)}{\partial x_i }
f'\big(u_\e(y)\big)
\Big(\frac{ \partial u_\e(y)}{\partial y_j}-\frac{ \partial u_0(y)}{\partial y_j}
\Big) dy.
\end{split}\end{equation}
Using that $f\in C^{1,\g}$ with $\gamma\in[0,1]$ we get
\begin{equation*}
\begin{split}
 \int_{\O_\e\backslash \widetilde{\O}_\e} \frac{\partial S(x,y)}{\partial x_i }\Big(
f'\big(u_\e(y)\big) -f'\big(u_0(y)\big)\Big)\frac{ \partial u_0(y)}{\partial y_j}dy
=O\left(\int_{\O_\e\backslash \widetilde{\O}_\e} \frac{|
 u_\e(y)-u_0(y)|^\g}{|x-y|^{N-1}}dy\right)=o\big(1\big).
\end{split}\end{equation*}
Now observe that by the fact $
C_1\e^{\beta}\leq  |x-P| \leq C_2\e^{\beta}$ and if $y\in \O_\e\backslash \widetilde{\O}_\e$ then
 $$\frac{|y-P|}\e\ge\frac{|x-P|^{\frac{N-\a}{N-1}}}\e\ge \e^{{\frac{\beta(N-\a)}{N-1}}-1}\to+\infty,$$
because $|x-P|\to0$. Hence, using \eqref{m18}, \eqref{a2019-11-15-05}, \eqref{2019-11-23-05} and \eqref{aa2019-11-25-10} we find
\begin{equation}\label{B21}
\begin{split}
 \frac{ \partial u_\e(y)}{\partial y_j}-\frac{ \partial u_0(y)}{\partial y_j}
 =O\left(\frac{\varepsilon^{N-2}}{|y-P|^{N-1}}\right)+o(1)=
 O\left(\frac{\varepsilon^{N-2}}{|x-P|^{N-\a}}\right)+o(1),~\mbox{for}~y\in \O_\e\backslash \widetilde{\O}_\e.
\end{split}\end{equation}
Finally we have
\begin{equation*}
\begin{split}
\int_{\O_\e\backslash \widetilde{\O}_\e} &\frac{\partial S(x,y)}{\partial x_i }
f'\big(u_\e(y)\big)
\left(\frac{ \partial u_\e(y)}{\partial y_j}-\frac{ \partial u_0(y)}{\partial y_j}
\right) dy \\&= O\left(\frac{\varepsilon^{N-2}}{|x-P|^{N-\alpha}}\right)+o(1)
= o\left(\frac{\varepsilon^{N-2}}{|x-P|^{N}}\right)+o(1).
\end{split}\end{equation*}
Hence the above estimates show
\begin{equation*}
\int_{\O_\e}\frac{\partial^2  S(x,y)}{\partial x_i\partial x_j}\big(p_\e(y)-p_\e(x)\big)dy=o\left(\frac{\varepsilon^{N-2}}{|y-P|^{N}}\right)+o(1),
\end{equation*}
which ends the estimate of $\overline{K}_0(w)$.
\vskip0.2cm
\noindent \underline{Estimate of $\overline{K}_1(w)$}\vskip0.2cm

Since by assumption $|x-P|=o(1)$ and $|w|=\frac{|x-P|}{\varepsilon}\rightarrow +\infty$,
we have that, for any $y\in \Omega_\varepsilon$,
\begin{equation*}
\begin{split}
\frac{\partial^2H_0\left(w,z\right)}{\partial w_i\partial w_j}=&
\frac{1}{N\omega_N} \left(
 \frac{|z|^2\delta_{ij}}{\left(|w|^2|z|^2-2\langle w,z\rangle+1\right)^\frac N2}-
 \frac{N(w_i|z|^2-z_i)(w_j|z|^2-z_j)}{\left(|w|^2|z|^2-2\langle w,z\rangle+1\right)^{\frac N2+1}}\right)\\ =
 &O\left(\frac{1}{|w|^N|z|^{N-2}}\right)
= O\left( \frac{\varepsilon^{2N-2}}{|x-P|^N |y-P|^{N-2}}\right).
\end{split}
\end{equation*}
 This gives us
\begin{equation}\label{2019-12-10-45}
\begin{split}
 \overline{K}_1(w)=&\frac1{\e^{N}}\int_{\Omega_\varepsilon}
 \frac{\partial^2H_0\left(w,\frac{y-P}\e\right)}{\partial w_i\partial w_j}p_\e(y)
dy \\=
&O\left( \frac{\varepsilon^{N-2}}{|x-P|^N}\int_{\Omega_\varepsilon}
 \Big(f\big(u_\varepsilon(y)\big)-f\big(u_0(y)\big)\right)
 \frac{1}{|y-P|^{N-2}}
dy\Big)
= o\left(\frac{\varepsilon^{N-2}}{|x-P|^N}\right).
\end{split}
\end{equation}
\vskip0.2cm
\noindent\underline{Estimate of $\overline{K}_2(w)$}\vskip0.2cm

We have that
\begin{equation}\label{2019-12-10-35}
\begin{cases}
\Delta_z\left(
      \frac{\partial^2 E_\varepsilon\left(w,z\right)}{\partial w_i\partial w_j}
      -\frac{\partial^2 H_0\left(w,z\right)}{\partial w_i\partial w_j}
\right)=0&~\mbox{for}~z\in \frac{\O-P}\e\setminus B(0,1),\\
 \frac{\partial^2 E_\varepsilon\left(w,z\right)}{\partial w_i\partial w_j}
 -\frac{\partial^2 H_0\left(w,z\right)}{\partial w_i\partial w_j}
 = 0&~\mbox{for}~z\in\partial B(0,1),\\
  \frac{\partial^2 E_\varepsilon\left(w,z\right)}{\partial w_i\partial w_j}
   -\frac{\partial^2H_0\left(w,z\right)}{\partial w_i\partial w_j}
= -\frac{\partial^2 H_0\left(w,z\right)}{\partial w_i\partial w_j}&~\mbox{for}~z\in\frac{\partial\O-P}\e.
\end{cases}
\end{equation}
Since $|w| \rightarrow +\infty$, arguing as in \eqref{2019-12-10-45} we have that
\begin{equation}\label{2019-12-10-36}
\frac{\partial^2 H_0\left(w,z\right)}{\partial w_i\partial w_j}=O\left(\frac{ \varepsilon^{N-2}}{|w|^N}\right)\,\,\, ~\mbox{for any}~z\in\frac{\partial\O-P}\e.
\end{equation}
Then by maximum principle, \eqref{2019-12-10-35} and \eqref{2019-12-10-36}  we get
\begin{equation}\label{2019-12-10-37}
\frac{\partial^2E_\varepsilon\left(w,z\right)}{\partial w_i\partial w_j}
 -\frac{\partial^2 H_0\left(w,z\right)}{\partial w_i\partial w_j}
=O\left(\frac{ \varepsilon^{N-2}}{|w|^N}\right)~\mbox{for any}~z\in\frac{\O-P}\e\backslash B(0,1).
\end{equation}
Hence \eqref{2019-12-10-41} and  \eqref{2019-12-10-37} imply
\begin{equation*}
\overline{K}_2(w)=O\left(\frac1{\e^2|w|^N}
 \int_{\Omega_\varepsilon}p_\e(y)dy\right)=o\left(\frac{\varepsilon^{N-2}}{|x-P|^N}\right),
\end{equation*}
which ends the estimate of $\overline{K}_2(w)$.
\vskip0.2cm
Collecting the estimates of $\overline{K}_0(w)$, $\overline{K}_1(w)$, $\overline{K}_2(w)$ by \eqref{2019-12-10-41}, the claim \eqref{2019-11-25-35} follows.
\end{proof}
Finally we consider the case $N=2$. As in the previous lemma we only consider a suitable neighborhood of $P$.
\begin{lemma}\label{B18}
Assume $N=2$ and   $\frac{C_1}{|\log\e|^{\delta}}\leq  |x-P| \leq \frac{C_2}{|\log\e|^{\delta}}$ for some $C_1,C_2,\delta>0$.  It holds
\begin{equation}\label{2019-11-25-35a}
\frac{\partial^2 A_\varepsilon(x)}{\partial x_i\partial x_j}=o\left(\frac{1}{|x-P|^{2}\cdot|\log \e|}\right)+o(1).
\end{equation}
\end{lemma}
\begin{rem}
Analogously to Remark \ref{re}, since critical points $x_\e$ satisfy $\frac{C_1}{|\log\e|} \leq  |x_\e-P| \leq \frac{C_2}{|\log\e|}$ ($\delta=1$ in Lemma \ref{B18}) our assumption is not restrictive.
\end{rem}
\begin{proof}[Proof of Lemma \ref{B18}]
Our starting point is again formula \eqref{2019-12-10-41} as $N=2$. As in the case $N\ge3$ we estimate the three terms.
\vskip0.2cm
\noindent\underline{Estimate of $\overline{K}_0(w)$}
\vskip0.2cm

As in Proposition \ref{B16} we have that
$p_\e(x)=o(1)$.
Now we estimate the integral in  $\overline{K}_0(w)$.
\begin{equation*}
\begin{split}
\int_{\O_\e} &\frac{\partial^2  S(x,y)}{\partial x_i\partial x_j}\big(p_\e(y)-p_\e(x)\big)dy
\\=& \int_{\O_\e\backslash \widetilde{\O}_\e} \frac{\partial^2  S(x,y)}{\partial x_i\partial x_j}\big(p_\e(y)-p_\e(x)\big)dy
+\int_{\widetilde{\O}_\e} \frac{\partial^2  S(x,y)}{\partial x_i\partial x_j}\big(p_\e(y)-p_\e(x)\big)dy,
\end{split}\end{equation*}
where $\widetilde{\O}_\e=B\big(P,|x-P|^{2-\alpha}\big)\backslash B(P,\e)$ and $\alpha\in(0,1)$.

Since
$\frac{\partial^2  S(x,y)}{\partial x_i\partial x_j}=O\left(\frac{1}{|x-P|^2}\right)$ for $y\in \widetilde{\O}_\e$
and $p_\e$ is bounded in $\Omega_\e$, we find
\begin{equation*}
 \int_{\widetilde{\O}_\e} \frac{\partial^2  S(x,y)}{\partial x_i\partial x_j}\big(p_\e(y)-p_\e(x)\big)dy
 =O\left(|x-P|^{2(1-\alpha)}\right)=o(1).
 \end{equation*}
Next, arguing as in \eqref{B19}, we know
\begin{equation*}
\begin{split}
\int_{\O_\e\backslash \widetilde{\O}_\e} &\frac{\partial^2  S(x,y)}{\partial x_i\partial x_j}\big(p_\e(y)-p_\e(x)\big)dy
\\=&
-\int_{\partial \big(\O_\e\backslash \widetilde{\O}_\e\big)} \frac{\partial  S(x,y)}{\partial x_i }\big(p_\e(y)-p_\e(x)\big)\nu_jd\sigma(y)
+\int_{\O_\e\backslash \widetilde{\O}_\e} \frac{\partial S(x,y)}{\partial x_i}\frac{\partial p_\e(y)}{\partial y_j}dy
\\=&
\underbrace{O\Big(|p_\e(x)|\Big) +O\Big( |x-P|^{\alpha}\Big)}_{=o(1)}+\int_{\O_\e\backslash \widetilde{\O}_\e} \frac{\partial S(x,y)}{\partial x_i}\frac{\partial p_\e(y)}{\partial y_j}dy.
\end{split}\end{equation*}
The last integral is estimated as in \eqref{B20},
\begin{equation*}
\begin{split}
\int_{\O_\e\backslash \widetilde{\O}_\e} &\frac{\partial S(x,y)}{\partial x_i}\frac{\partial p_\e(y)}{\partial y_j}dy\\=&
\underbrace{O\left(\int_{\O_\e\backslash \widetilde{\O}_\e} \frac{|
 u_\e(y)-u_0(y)|^\g}{|x-y|^{N-1}}dy\right)}_{=o(1)}+
\int_{\O_\e\backslash \widetilde{\O}_\e} \frac{\partial S(x,y)}{\partial x_i }
f'\big(u_\e(y)\big)
\Big(\frac{ \partial u_\e(y)}{\partial y_j}-\frac{ \partial u_0(y)}{\partial y_j}
\Big) dy.
\end{split}\end{equation*}
Finally since $\frac{|y-P|}\e\ge\frac{|x-P|^{2-\a}}\e\ge\frac C{\e|\log\e|^{(2-\a)\delta}}\to\infty$
 as in \eqref{B21}, we get for $y\in \O_\e\backslash \widetilde{\O}_\e$,
\begin{equation*}
\begin{split}
 \frac{ \partial u_\e(y)}{\partial y_j}-\frac{ \partial u_0(y)}{\partial y_j}
 =O\left(\frac{1}{|y-P|\cdot|\log \e|}\right)+o(1)
 = O\Big(\frac{1}{|x-P|^{2-\alpha}\cdot|\log \e|}\Big)+o(1),
\end{split}\end{equation*}
and
\begin{equation*}
\begin{split}
\int_{\O_\e\backslash \widetilde{\O}_\e} &\frac{\partial S(x,y)}{\partial x_i }
f'\big(u_\e(y)\big)
\left(\frac{ \partial u_\e(y)}{\partial y_j}-\frac{ \partial u_0(y)}{\partial y_j}
\right) dy \\&= O\Big(\frac{1}{|x-P|^{2-\alpha}\cdot|\log \e|}\Big)+o(1)
=o\Big(\frac{1}{|x-P|^{2}\cdot|\log \e|}\Big)+o(1).
\end{split}\end{equation*}
Hence the above estimates show
\begin{equation}\label{B24}
\overline{K}_0(w)=\int_{\O_\e}\frac{\partial^2  S(x,y)}{\partial x_i\partial x_j}\big(p_\e(y)-p_\e(x)\big)dy+o(1)=o\left(\frac{1}{|x-P|^{2}\cdot|\log \e|}\right)+o(1).
\end{equation}
\vskip0.2cm
\noindent \underline{Estimate of $\overline{K}_1(w)$}
\vskip0.2cm
We have that
\begin{equation*}
\begin{split}
 \frac{\partial^2H_0\left(w,z\right)}{\partial w_i\partial w_j}=&
\frac{1}{2\pi} \left(
 \frac{|z|^2\delta_{ij}}{|w|^2|z|^2-2\langle w,z\rangle+1}-
 \frac{2(w_i|z|^2-z_i)(w_j|z|^2-z_j)}{\left(|w|^2|z|^2-2\langle w,z\rangle+1\right)^{2}}\right)\\=&
 \frac{1}{2\pi} \Big(
  \frac{ \delta_{ij}}{|w|^2}-
  \frac{2w_iw_j}{|w|^4}\Big)+O\Big(\frac{1}{|w|^3|z|}\Big),
\end{split}
\end{equation*}
and then
\begin{equation*}
\begin{split}
\overline{K}_1(w)=&\frac1{\e^2}\int_{\Omega_\varepsilon}
 \frac{\partial^2H_0\left(w,\frac{y-P}\e\right)}{\partial w_i\partial w_j}p_\e(y)
dy\\=
&\frac{1}{2\pi} \Big(\frac{ \delta_{ij}}{|x-P|^2}-\frac{2(x_i-P_i)
 (x_j-P_j)}{|x-P|^4}\Big)
  \int_{\Omega_\varepsilon}p_\varepsilon(y)
 dy+O\Big(\frac{\e^3}{|x-P|^3}\Big).
\end{split}
\end{equation*}
We will see that this term will cancel with the main term of $\overline{K}_2(w)$.
\vskip0.2cm
\noindent\underline{Estimate of $\overline{K}_2(w)$}
\vskip0.2cm
Let us introduce the function
$$
\hat{K}_2(w,z):=
   \Big(\frac{\partial^2 E_\varepsilon\left(w,z\right)}{\partial w_i\partial w_j}
   -\frac{\partial^2 H_0\left(w,z\right)}{\partial w_i\partial w_j}\Big)
-\frac{1}{2\pi} \Big(\frac{\delta_{ij}}{|w|^2}-\frac{2w_iw_j}{|w|^4}\Big)\frac{\log
|z|}{\log\varepsilon},
$$
which verifies
 \begin{equation*}
\begin{cases}
\Delta_z\hat{K}_2(w,z)=0&~\mbox{for}~z\in \frac{\O-P}\e\setminus B(0,1),\\
\hat{K}_2(w,z)=0&~\mbox{for}~z\in\partial B(0,1),\\
\hat{K}_2(w,z)=
   -\frac{\partial^2 H_0\left(w,z\right)}{\partial w_i\partial w_j}
-\frac{1}{2\pi} \Big(\frac{ \delta_{ij}}{|w|^2}-\frac{2w_iw_j}{|w|^4}\Big)\frac{\log
|z|}{\log\varepsilon}
&~\mbox{for}~z\in\frac{\partial\O-P}\e.
\end{cases}
\end{equation*}
Then for any $z\in\frac{\partial\O-P}\e$, we have that
\begin{equation*}
   -\frac{\partial^2 H_0\left(w,z\right)}{\partial w_i\partial w_j}
-\frac{1}{2\pi} \Big(\frac{\delta_{ij}}{|w|^2}-\frac{2w_iw_j}{|w|^4}\Big)\frac{\log
|z|}{\log\varepsilon}
=O\Big(\frac{1}{|w|^2\cdot |\log \varepsilon|}\Big).
\end{equation*}
Hence by the maximum principle for harmonic function, for any $y\in \O_\varepsilon$
\big(or $z\in \frac{ \O-P}\e \backslash B(0,1)$\big),
  we deduce that
 \begin{equation*}
 \begin{split}
\frac{\partial^2 E_\varepsilon\left(w,z\right)}{\partial w_i\partial w_j}
 -\frac{\partial^2 H_0\left(w,z\right)}{\partial w_i\partial w_j}
=&
 \frac{1}{2\pi} \Big(\frac{ \delta_{ij}}{|w|^2}-\frac{2w_iw_j}{|w|^4}\Big)\frac{\log
  |z|}{\log\varepsilon}+
O\Big(\frac{\varepsilon^2}{|x-P|^2\cdot|\log \varepsilon|}\Big)\\=& -\frac{1}{2\pi} \Big(\frac{ \delta_{ij}}{|w|^2}-\frac{2w_iw_j}{|w|^4}\Big)
+O\Big(\frac{\varepsilon^2}{|x-P|^2\cdot|\log \varepsilon|}\Big).
   \end{split}\end{equation*}
Then we obtain
 \begin{equation*}
 \begin{split}
\overline{K}_2(w)=&\frac1{\e^{2}}\int_{\Omega_\varepsilon}
\Big(\frac{\partial^2G_0\left(w,z\right)}{\partial w_i\partial w_j}
-\frac{\partial^2 S\left(w,z\right)}{\partial w_i\partial w_j}\Big) \Big|_{z=\frac{y-P}
{\varepsilon}}p_\varepsilon(y)dy  \\=&-\frac{1}{2\pi} \Big(\frac{\delta_{ij}}{|x-P|^2}-\frac{2(x_i-P_i)
 (x_j-P_j)}{|x-P|^4}\Big)
  \int_{\Omega_\varepsilon}p_\varepsilon(y)
 dy+o\left(\frac{1}{|x-P|^2\cdot|\log \varepsilon|}\right),
\end{split}\end{equation*}
which ends the estimate of $\overline{K}_2(w)$.

Now we observe that
 \begin{equation*}
\overline{K}_1(w)+\overline{K}_2(w)=o\left(\frac{1}{|x-P|^2\cdot|\log \varepsilon|}\right)
\end{equation*}
which jointly with \eqref{B24} gives \eqref{2019-11-25-35a}. This ends the proof.
\end{proof}

Now we are in position to give the estimate of the second derivative of $u_\e$.
\begin{prop}\label{B35}
Let  $u_0$ and $u_\varepsilon$ solutions to \eqref{1} and \eqref{aa2} respectively. Then we have the following estimates,
\vskip 0.2cm\noindent
If  $N\geq 3$ and $
C_1\e^{\beta}\leq  |x-P| \leq C_2\e^{\beta}$
for some $C_1,C_2>0$ and $\beta\in (0,\frac{N-1}{N})$, then it holds
\begin{equation}\label{c2019-11-15-01}
\frac{\partial^2 u_\e(x)}{\partial x_i\partial x_j}=\frac{\partial^2 u_0(x)}{\partial x_i\partial x_j}
+o(1)+
\frac{(N-2)u_0(P)}{|x-P|^N}\varepsilon^{N-2}\left(\delta_{ij}-\frac{N(x_i-P_i)(x_j-P_j)}
{|x-P|^2}+o\big(1\big)\right).
\end{equation}
If $N=2$ and $\frac{C_1}{|\log\e|^{\delta}}\leq  |x-P| \leq \frac{C_2}{|\log\e|^{\delta}}$ for some $C_1,C_2,\delta>0$, then it holds
\begin{equation}\label{d2019-11-15-01}
\frac{\partial^2 u_\e(x)}{\partial x_i\partial x_j}=\frac{\partial^2 u_0(x)}{\partial x_i\partial x_j}
+o(1)+
\frac{ u_0(P)}{|\log \varepsilon| \cdot |x-P|^2} \left(\delta_{ij}-\frac{2(x_i-P_i)(x_j-P_j)}{|x-P|^2}+o\left(\frac{1}{|x-P|}\right)\right).
\end{equation}
\end{prop}
\begin{proof}
If $N\ge3$ we have that  \eqref{c2019-11-15-01} follows by \eqref{2019-11-15-05b}, \eqref{B15}
and \eqref{2019-11-25-35}. If $N=2$ we have that \eqref{d2019-11-15-01} follows by \eqref{2019-11-15-05b}, \eqref{B15}
and \eqref{2019-11-25-35a}.
\end{proof}
\vfill\eject
\section{ Proofs of main theorems when $\nabla u_0(P)\neq 0$}\label{s5}

In this section we prove our results when  $\nabla u_0(P)\neq 0$. Our first proposition is quite known but we did not find any reference. So we give a complete proof.
\begin{prop}\label{prop3.1}
Assume that $x_0$ is a nondegenerate critical point of $v\in C^2\big(B(x_0,1)\big)$
and $v_\e\in C^2\big(B(x_0,1)\big)$ verifies that $v_\e\to v$ in $C^2\big(B(x_0,1)\big)$. Then
 $v_\e$ has  a unique critical point $x_\e$ in $B(x_0,r)$ which is also nondegenerate for $r$ small enough.
\end{prop}
\begin{proof}
It is not restrictive to suppose that $index_{x_0}(\nabla u_0)=1$. Since $x_0$ is a nondegenerate critical point of $u$ then it is isolated in $B(x_0,r)$ for $r$ small enough. Moreover by the convergence of $u_\e$ to $u_0$ we get
\begin{equation}\label{deg}
deg\big(\nabla u_\e,0,B(x_0,r)\big)=deg\big(\nabla u_0,0,B(x_0,r)\big)=index_{x_0}(\nabla u_0)=1.
\end{equation}
This gives the existence of a critical point $x_\e$ of $u_\e$ and of course $x_\e\to x_0$. Let us show the uniqueness of the critical point $x_\e$. By the $C^2$ convergence of $u_\e$ to $u_0$ we get that any critical point $x_\e$ is nondegenerate and therefore if $\mathcal{K}_\e=\{x\in B(x_0,r):\nabla u_\e(x)=0\}$ we have that $\sharp\mathcal{K}_\e=n_\e<+\infty$. Moreover again by the $C^2$ convergence of $u_\e$ to $u_0$ we get that $index_{x_\e}(\nabla u_0)=1$. Finally we have that
$$deg\big(\nabla u_\e,0,B(x_0,r)\big)=\sum_{x\in\mathcal{K}_\e}index_{x}(\nabla u_\e)=n_\e$$
which jointly with \eqref{deg}  gives $n_\e=1$. This proves the uniqueness of $x_\e$ .
\end{proof}
\begin{rem}\label{remdeg}
Using the same proof of the previous proposition it is possible to prove that if a smooth vector field $V:B(x_0,1)\subset\R^N\to\R^N$ verifies $V(x_0)=0$ and $\det Jac\big(V(x_0)\big)\ne0$ then any approximating vector field $V_\e:B(x_0,1)\subset\R^N\to\R^N$ such that $V_\e\to V$ in $C^1\big(B(x_0,1)\big)$ admits a unique zero $x_\e$ such that $x_\e\to x_0$.
\end{rem}

Next lemma will be also useful.
\begin{lemma}\label{G9}
Let us consider the matrix $B=\left(\delta_{ij}-N\frac{\xi_i\xi_j}{|\xi|^2}\right)_{1\leq i,j\leq N}$, $\xi=(\xi_1,..,\xi_N)\ne0$ and $N\ge2$. Then we have that the eigenvalues of $B$ are given by
\begin{equation*}
\lambda_1=1-N,\quad\lambda_2=..=\lambda_N=1.
\end{equation*}
\end{lemma}
\begin{proof}
We have that the vector $v_1=\xi$ is the first eigenvector of $B$ and a straightforward computation shows that $\lambda_1=1-N$. Concerning the other eigenvalues, since $\xi\ne0$ we can assume that
$\xi\ne0$. Then the vector $v=\left(\frac{\sum_{j=2}^N\xi_jt_j}{\xi_1},t_2,..,t_N\right)$ with $t_2,..,t_N$ real numbers, is an eigenvector with multiplicity $N-1$ of the matrix $B$. Again a direct computation shows that $\lambda_2=..=\lambda_N=1.$
\end{proof}
Let us denote by $\mathcal{C}$ the set of critical point of $u_0$, i.e.
\begin{equation*}
\mathcal{C}=\Big\{x\in\O \hbox{ such that }\nabla u_0(x)=0\Big\}.
\end{equation*}
Since $\nabla u_0(P)\ne0$ we get that there exists $d>0$ such that $\mathcal{C}\bigcap B(P,d)=\varnothing$.

\smallskip

\begin{proof}[\textbf{Proof of Theorem \ref{I7}}]
Let us define the vector field
$F(y)=\big(F_1(y),\cdots,F_N(y)\big)$ on $W$ as
\begin{equation*}
F_i(y)=
\begin{cases}\frac{\partial u_0(P)}{\partial x_i}+(N-2)u_0(P) \frac{y_i}{|y|^{N}}~&\mbox{if}~N\geq 3,\\
\frac{\partial u_0(P)}{\partial x_i}+\frac{y_i}{2|y|^2}u_0(P)~&\mbox{if}~N=2,
\end{cases}
\end{equation*}
and
 \begin{equation*}
W=
\begin{cases}B\left(0,\frac d{\e^\frac{N-2}{N-1}}\right)\backslash B\left(0,\varepsilon^\frac1{N-1}\right)~&\mbox{if}~N\geq 3,\\
B\left(0, d|\ln \e|\right)\backslash B\left(0,\varepsilon|\ln \e|\right)~&\mbox{if}~N=2.
\end{cases}
\end{equation*}
We have that, for $C_N, C_2$ as in \eqref{I23},
\begin{equation*}
y_0=
\begin{cases}
C_N\nabla u_0(P)&\mbox{for}~N\geq 3,\\
C_2\nabla u_0(P) &\mbox{for}~N=2,
\end{cases}
\end{equation*}
is the unique zero of the vector field $F$. Moreover the index of $F$ at $y_0$ is given by $index_{y_0}F=sgn\Big(\det Jac\big(F(y_\e)\big)\Big)$ and
\begin{equation*}
\begin{split}
 \det Jac\big(F(y_\e)\big)=&\det\left[-\frac1{C_N}\left(\delta_{ij}-N\frac{\xi_i\xi_j}{|\xi|^2}\right)_{1\leq i,j\leq N}\right]\\=&
 \left(-\frac1{C_N}\right)^N\displaystyle\Pi^N_{i=1}\la_i=\left(-\frac1{C_N}\right)^N(1-N)<0,
\end{split}
\end{equation*}
where  $\la_i$($i=1,..,N$)  are the eigenvalues of the matrix $\left(\delta_{ij}-N\frac{\xi_i\xi_j}{|\xi|^2}\right)_{1\leq i,j\leq N}$.
Hence
\begin{equation}\label{ind}
index_{y_0}F=-1.
\end{equation}
Next let us introduce the vector field
$F_\varepsilon(y)=\big(F_{\varepsilon,1}(y),\cdots,F_{\varepsilon,N}(y)\big)$ on $W$
as
$$F_{\varepsilon,i}(y)=\begin{cases}\frac{\partial u_\e}{\partial x_i}\big(P+\e^\frac{N-2}{N-1}y\big)~&\mbox{if}~N\geq 3,\\
\frac{\partial u_\e}{\partial x_i}\big(P+ \frac{1}{|\ln \e|}y\big)~&\mbox{if}~N=2.
\end{cases}
$$
By Proposition \ref{B10} we have that $F_{\varepsilon,i}(y)=0$ if and only if  $y_\e=y_0+o(1)$. Moreover by Propositions \ref{prop2.4} and \ref{B35} we have that $F_\e\to F$ in $C^1(B(y_0,1).$ Hence Remark \ref{remdeg} applies and then $F_\varepsilon(y)$  has a unique zero $y_\e$ close to $y_0$ in a ball $B(y_0,r)$. Using again Proposition \ref{B10} we get that $y_\e$ is the unique zero of $F_\varepsilon(y)$ in $W$. This is equivalent to say that
\begin{equation*}
x_\e=P+
\begin{cases}
 \e^\frac{N-2}{N-1}\big(y_0+o(1)\big)&~\mbox{for}~N\geq 3,\\
 \frac{1}{|\log \varepsilon|}\big(y_0+o(1)\big)&~\mbox{for}~N=2,
\end{cases}
\end{equation*}
 is the unique zero in $B(P,d)\backslash B(P,\varepsilon)$ which proves \eqref{2019-11-27-03}. Moreover  by \eqref{ind} we obtain that the index of $\nabla u_\e$ at $x_\e$ is $1$ and the uniqueness of $x_\e$ imply that $x_\e$ is a saddle point. Finally \eqref{I16} follows by Proposition \ref{B16}.
\end{proof}
\begin{proof} [\textbf{Proof of Thereom \ref{th1-1}}]
We write $\Omega_\varepsilon=\mathcal{K}_1\bigcup \mathcal{K}_2 \bigcup \Big(B(P,d)\backslash B(P,\varepsilon)\Big)$ with $$\mathcal{K}_1=\{x,~dist~(x,\mathcal{C})\leq d\}~\mbox{and}~\mathcal{K}_2:=\Omega_\varepsilon\backslash \big(\mathcal{K}_1\bigcup B(P,d)\big),$$
for some small  fixed $d>0$.

First, Proposition \ref{prop3.1} gives us that
$$\sharp\{\mbox{critical points of $u_\varepsilon$ in $\mathcal{K}_1$}\}=\sharp\{\mathcal{C}\},$$
and from Theorem \ref{I7}, we get that
$$\sharp\{\mbox{critical points of $u_\varepsilon$ in $B(P,d)\backslash B(P,\varepsilon)$}\}=1.$$
Finally, using Lemma \ref{lemma-1} we find that
$$\sharp\{\mbox{critical points of $u_\varepsilon$ in $\mathcal{K}_2$}\}=0.$$
Then from the above discussion, we get \eqref{mainresult}.
Finally by Theorem \ref{I7} we get that  the additional critical point $x_\e$ of $u_\e$ is a saddle point of index $-1$ satisfying \eqref{2019-11-27-03}.
\end{proof}
\vfill\eject

\section{Proofs of the main results when $\nabla u_0(P)=0$ }\label{s7}
In this section we consider the case  $\nabla u_0(P)=0$. Our first aim is to improve the estimate \eqref{2019-11-15-01} . More precisely we want to replace the term $\frac{\partial u_0(x)}{\partial x_i}+o(1)$ with $\sum^N_{j=1}\left(\frac{\partial^2 u_0(P)}{\partial x_i\partial x_j}+o(1)\right)(x_j-P_j)$. This can be done by using that  $\nabla u_0(P)=0$ and Taylor's formula if we show that the term $o(1)$ (coming from \eqref{aa2019-11-25-10}) can be improved to $o(|x-P|)$. This will be done in next lemma, where the condition \eqref{2019-12-26-01} is used.
\begin{lemma}\label{bdL1}
 If $|x-P|\varepsilon^{-1}\rightarrow\infty$, $|x-P|\rightarrow 0$ and
 the nonlinearity $f$ satisfies  the condition \eqref{2019-12-26-01}, then it holds
\begin{equation}\label{2019-11-25-10}
\frac{\partial A_\varepsilon(x)}{\partial x_i}=o\Big(|x-P|\Big)+
\begin{cases}
o\Big(\frac{\varepsilon^{N-2}}{|x-P|^{N-1}}\Big)&~\mbox{for}~N\geq 3,\\
o\Big(\frac{1}{|x-P|\cdot|\log\varepsilon|}\Big)&~\mbox{for}~N=2.
\end{cases}
\end{equation}
\end{lemma}
\begin{proof}
For $N\geq 3$, from the estimate \eqref{2019-12-26-02} we have that
\begin{equation}\label{2020-02-02-10}
\begin{split}
&\frac{\partial  A_\varepsilon(x)}{\partial x_i}=   \int_{\Omega_\varepsilon}\Big(f\big(u_\varepsilon(y)\big)-f\big(u_0(y)\big)\Big)
\frac{\partial G(x,y)}{\partial x_i}dy+o\left(\frac{\varepsilon^{N-2}}{|x-P|^{N-1}}+\varepsilon\right).
 \end{split}\end{equation}
Then combining \eqref{2019-12-26-01} and \eqref{2020-02-02-10}, we find  \eqref{2019-11-25-10} for $N\geq 3$.

Similarly, for $N=2$, from the estimate \eqref{2019-12-26-02ad} we have
 \begin{equation}\label{2019-12-26-02af}
\begin{split}
 \frac{\partial  A_\varepsilon(x)}{\partial x_i}=   \int_{\Omega_\varepsilon}\Big(f\big(u_\varepsilon(y)\big)-f\big(u_0(y)\big)\Big)
\frac{\partial G(x,y)}{\partial x_i}dy+o\left(\frac{1}{|x-P|\cdot|\log\varepsilon|}\right).
 \end{split}\end{equation}
 Then combining \eqref{2019-12-26-01} and \eqref{2019-12-26-02af}, we find  \eqref{2019-11-25-10} for $N=2$.
\end{proof}
\vskip0.2cm
A consequence of Lemma \ref{bdL1}  is that we can rewrite \eqref{2019-11-15-01} as follows,
\begin{equation}\label{C10}
\frac{\partial u_\e(x)}{\partial x_i}=\sum^N_{j=1}\left(\frac{\partial^2 u_0(P)}{\partial x_i\partial x_j}+o(1)\right)(x_j-P_j)+
 \begin{cases}
(N-2)\big(u_0(P)+o(1)\big)\frac{x_i-P_i}{|x-P|^N}\e^{N-2}&~\mbox{if}~N\geq 3,\\
\big(u_0(P)+o(1)\big)\frac{x_i-P_i}{ |x-P|^2}\frac{1}{|\log\e|}&~\mbox{if}~N=2.
\end{cases}
\end{equation}
From now we assume that
\vskip0.2cm
\centerline{
$P$\em{ is a nondegenerate critical point of the solution $u_0$.}
}
\vskip0.2cm
Next proposition states a necessary condition for critical points of $u_\e$.
\begin{prop}
If $\nabla u_0(P)=0$ and $x_\e$ is a critical point of $u_\e$ such that $x_\e\to P$ as $\e\to0$ then we have that
\begin{equation}\label{d2019-11-16-01}
x_\varepsilon=P+
\begin{cases}
\left[\frac{ (2-N)u_0(P)+o(1)}{\la}\right]^\frac1N
\varepsilon^{\frac{N-2}{N}}v~&\mbox{if}~N\geq 3,\\
\sqrt{-\frac{u_0(P)+o(1)}\la}\frac{1}{\sqrt{|\log\e|}}v~&\mbox{if}~N=2,
\end{cases}
\end{equation}
where  $\lambda$ is a negative  eigenvalue of the matrix $\textbf{H}(P)$ and $v$ an associated eigenfunction with $|v|=1$.
\end{prop}
\begin{proof}
By \eqref{C10} we get that if $\frac{\partial u_\e}{\partial x_i}(x_\e)=0$ then
\begin{equation}\label{2019-11-16-01}
0=\sum^N_{j=1}\left(\frac{\partial^2 u_0(P)}{\partial x_i\partial x_j}+o(1)\right)(x_{j,\e}-P_j)+
 \begin{cases}
(N-2)\big(u_0(P)+o(1)\big)\frac{x_{i,\e}-P_i}{|x_\e-P|^N}\e^{N-2}&~\mbox{if}~N\geq 3,\\
\big(u_0(P)+o(1)\big)\frac{x_{i,\e}-P_i}{ |x_\e-P|^2}\frac{1}{|\log\e|}&~\mbox{if}~N=2.
\end{cases}
\end{equation}
By \eqref{2019-11-16-01} we immediately get that, as $\e\to0$,
\begin{itemize}
\item if $N\ge3$ then $(2-N)u_0(P)\frac{\e^{N-2}}{|x_\e-P|^N}\to\lambda$.
\item if $N=2$ then $-u_0(P)\frac{1}{ |x_\e-P|^2|\log\e|}\to\lambda$.
\end{itemize}
Since  $det~\textbf{H}(P)\neq 0$ we have that  $\lambda$ is a negative eigenvalue of the matrix $\textbf{M}(P)$. Dividing \eqref{2019-11-16-01} by $|x-x_\e|$ and passing to the limit the claim \eqref{d2019-11-16-01} follows.
\end{proof}
By the previous proposition we get the following corollary.
\begin{cor}\label{C11}
Let us suppose that $P$ is a nondegenerate minimum point of $u_0$ and $B(P,d)$ is a ball such that $\nabla u_0\ne0$ for any $x\in B(P,d)\setminus\{P\}$. Then there is no critical point of $u_\e$ in $B(P,d)\setminus B(P,\e)$.
\end{cor}
Here we are in position to give the proofs of Theorem \ref{th1-2}.

\smallskip

\begin{proof}[\textbf{Proof of Theorem \ref{th1-2}}]~Since $P$ is a nondegenerate critical point of $u_0$ from  the $C^2$ convergence of $u_\e$ to $u_0$ we find that
\begin{equation}\label{2019-11-27-06}
\sharp\{\mbox{critical points of $u_\varepsilon(x)$ in $\Omega\backslash B(P,d)$}\}= \sharp\{\mathcal{C}\}-1,
\end{equation}
where $d>0$ is small fixed constant with $\mathcal{C}\bigcap \Big(B(P,d)\backslash B(P,\varepsilon)\Big)=\varnothing$.

%Moreover by Proposition \ref{B10} we have that if $x_\e$ is a critical point of $u_\e$ in $ \Big(B(P,d)\backslash B(P,\varepsilon)$ then \eqref{d2019-11-16-01} holds.
Let us introduce the vector field $\widehat{F}(y)=\big(\widehat{F}_{1}(y),..,\widehat{F}_{N}(y)\big)$ on $\widetilde{W}$ as
\begin{equation*}
\widehat{F}_{i}(y):=\sum^N_{j=1}\frac{\partial^2 u_0(P)}{\partial x_i\partial x_j}y_j+
 \begin{cases}
(N-2)u_0(P)\frac{y_i}{|y|^N} &~\mbox{if}~N\geq 3,\\
u_0(P)\frac{y_i}{ |y|^2} &~\mbox{if}~N=2,
\end{cases}
\end{equation*}
and
\begin{equation*}
\widetilde{W}=
\begin{cases}B\left(0,\frac d{\e^\frac{N-2}{N}}\right)\backslash B\left(0,\varepsilon^\frac2{N}\right)~&\mbox{if}~N\geq 3,\\
B\left(0, d\sqrt{|\ln \e|}\right)\backslash B\left(0,\varepsilon\sqrt{|\ln \e|}\right)~&\mbox{if}~N=2.
\end{cases}
\end{equation*}
Suppose that $\textbf{H}(P)$ has $m$ negative eigenvalues
$$\lambda_1\leq \lambda_{2} \leq \cdots \leq \lambda_m<0,$$
with associated eigenfuctions $v_1,..,v_m$ satisfying $|v_i|=1$ for $i=1,..,m$.
By definition, if $\widehat{F}(\bar{y})=0$ we have that for $j\in\{1,..,m\}$,
%\begin{equation}\label{B-26}
%|x-P|=
%\begin{cases}
%\left(\frac{\la_i\e^{N-2}}{(2-N)u_0(P)}\right)^\frac1N,&~\mbox{if}~N\geq 3,\\
%\sqrt{\frac{u_0(P)}{\la_i}}\frac1{\sqrt{|\log\e|}},&~\mbox{if}~N=2,
%\end{cases}
%\end{equation}
%and then
\begin{equation}\label{B26}
\bar{y}=y^{(j)}_\pm=\pm r_jv_j~\mbox{with}~ r_j=
\begin{cases}
\left(\frac{(2-N)u_0(P)}{\la_j}\right)^\frac1N&~\mbox{if}~N\geq 3,\\
\sqrt{-\frac{u_0(P)}{\la_j}}&~\mbox{if}~N=2.
\end{cases}
\end{equation}
Let us compute the index of the vector field $\widehat{F}$ at $=y^{(j)}_\pm$. We have that, denoting by $v_j=(v_{j1},..,v_{jN})$,
$$\textbf{M}_{ik}\left(y^{(j)}_\pm\right)=\frac{\partial\widehat{F}_i\left(y^{(j)}_\pm\right)}{\partial y_k}=\frac{\partial^2 u_0(P)}{\partial x_{i}\partial x_{k}}-\lambda_i\delta_{ik}+N\lambda_iv_{ji}v_{jk}.$$
We compute the determinant of the matrix $\textbf{M}(\bar{y})$ by writing its eigenvalues. We have that,
\begin{itemize}
\item Every eigenvector $v_n\ne v_j$ of $\textbf{H}(P)$,  $n=1,..,N$   is an eigenvector of $\textbf{M}(\bar{y})$ with eigenvalues $\la_n-\la_j$.
\item The eigenvector $v_j$ of $\textbf{H}(P)$  is an eigenvector of $\textbf{M}(\bar{y})$ with eigenvalue $N\la_j$.
\end{itemize}
Hence we have that
\begin{equation}\label{B27}
\det\textbf{M}(\bar{y})=N\la_j\prod_{n\ne j}(\lambda_n-\lambda_j).
\end{equation}
Note that if all the negative eigenvalues of $\textbf{H}(P)$ are $simple$ then $\det\textbf{M}(\bar{y})\ne0$.
\vskip0.2cm

Now we consider two different cases.
\vskip0.2cm
\noindent {\underline{Case 1} (negative eigenvalues of  $\textbf{H}(P)$ are simple)}
\vskip0.2cm

For any $j=1,\cdots,m$, there exists a constant $d_1>0$ such that $\widehat{F}(y)$ has exactly one zero in $B\left(y^{(j)}_+,d_1\right)$ and $B\left(y^{(j)}_-,d_1\right)$ respectively. Moreover by \eqref{B27} we have that $y_+$ and $y_-$ are $non$-$singular$ zeros. Furthermore
$\widehat{F}(y)$ has no solutions in $\widetilde{W} \backslash\bigcup^{m}_{j=1}B\left(y^{(j)}_\pm,d_1\right)$.

Next let us introduce the vector field
$\widetilde{F}_\varepsilon(y)=\big(\widetilde{F}_{\varepsilon,1}(y),\cdots,\widetilde{F}_{\varepsilon,N}(y)\big)$ on $\widetilde{W}$
as
$$\widetilde{F}_{\varepsilon,i}(y)=\begin{cases}\frac{\partial u_\e}{\partial x_i}\left(P+\e^\frac{N-2}{N}y\right)~&\mbox{if}~N\geq 3,\\
\frac{\partial u_\e}{\partial x_i}\left(P+ \frac{1}{\sqrt{|\ln \e|}}y\right)~&\mbox{if}~N=2.
\end{cases}
$$
Moreover by Proposition \ref{prop2.4} and Proposition \ref{B35} (which we apply when $\beta=\frac{N-2}{N}$ and $\delta=\frac{1}{2}$) we have that $\widetilde{F}_\e\to \widetilde{F}$ in $C^1\big(B(\pm r_jv_j,d_1)\big)$. Hence Remark \ref{remdeg} applies and then $\widetilde{F}_\varepsilon(y)$  has a unique solution $y^{(j)}_{\e,\pm}\to y^{(j)}_\pm$ in $B\left(y^{(j)}_\pm,d_1\right)$.

Furthermore by Proposition \ref{prop2.4}  $\widehat{F}_\e(y)$ has no solutions in $\widetilde{W} \backslash\bigcup^{m}_{j=1}B\left(y^{(j)}_\pm,d_1\right)$.
Hence $\widetilde{F}_\varepsilon(y)$ has $2m$ zeros $y^{(j)}_{\e,\pm}$ and then $u_\e(x)$ has $2m$ critical points in $B(P,d)\backslash B(P,\varepsilon)$ satisfying
\eqref{I21}. Finally by \eqref{2019-11-27-06} we get \eqref{adda2019-11-27-07} which
ends the proof of $(1)$ in Theorem \ref{th1-2}.
 \vskip0.2cm

 \noindent
{\underline{Case 2} (negative eigenvalues of  $\textbf{H}(P)$ are multiple)}
\vskip 0.2cm
In this case we have by \eqref{B27} that $\det\textbf{M}(\bar{y})=0$. Since  the degree theory seems  difficult to use in this case, we will prove the claim showing directly that there are at least $2m$ zeros for $\nabla u_\e$.

Let us denote by $\textbf{Q}$ the orthogonal matrix such that
$$\textbf{Q}^T\textbf{H}(P)\textbf{Q}=diag(\la_1,..,\la_N).$$
Hence, denoting by $Y=\textbf{Q}^T(x-P)$ we get that \eqref{2019-11-16-01} becomes
\begin{equation*}
0=\big(\la_i+o(1)\big)Y_i+
 \begin{cases}
(N-2)\big(u_0(P)+o(1)\big)\frac{Y_i}{|Y|^N}\e^{N-2}&~\mbox{if}~N\geq 3,\\
-\big(u_0(P)+o(1)\big)\frac{Y_i}{ |Y|^2}\frac{1}{|\log\e|}&~\mbox{if}~N=2.
\end{cases}
\end{equation*}
Let us consider the case $N\ge3$ ($N=2$ can be managed in the same way) and introduce the points
$Y_{1,\e}=
\big((1-b)r_1\e^\frac{N-2}N,0,..,0\big)$ and $Y_{2,\e}=\big((1+b)r_1\e^\frac{N-2}N,0,..,0\big)$ for $b\in(0,1)$ and $r_1$ as in \eqref{B26}.
We have that for any $i=2,..,N$ it holds $\frac{\partial u_\e(Y_{1,\e})}{\partial y_i}=\frac{\partial u_\e(Y_{2,\e})}{\partial y_i}=0$ and

\begin{equation*}
\begin{split}
\frac{\partial u_\e(Y_{1,\e})}{\partial y_1}=&\textbf{Q}^T\frac{\partial u_\e(Y_{1,\e})}{\partial x_1}=
\big(\la_1+o(1)\big)Y_{1,\e}+(N-2)\big(u_0(P)+o(1)\big)\frac{Y_{1,\e}}{|Y_{1,\e}|^N}\e^{N-2}\\=
&(1-b)r_1\la_1\e^\frac{N-2}N\left[1-\frac1{(1-b)^N}+o(1)\right]>0,
 \end{split}\end{equation*}
and analogously
$$\frac{\partial u_\e(Y_{2,\e})}{\partial y_1}=(1-b)r_1\la_1\e^\frac{N-2}N\left[1-\frac1{(1+b)^N}+o(1)\right]<0.$$
This implies that there exists $Y_\e=\big(A\e^\frac{N-2}N,0,..,0\big)$ with $A\in\big((1-b)r_1,(1+b)r_1\big)$ such that $\nabla u_\e(Y_\e)=0$. Of course the same computation holds if we replace
$Y_{1,\e},Y_{2,\e}$ by $-Y_{1,\e},-Y_{2,\e}$ getting the existence of a $second$ critical point.
Finally repeating this argument for any negative eigenvalue we get the existence of at least $2m$ critical points which ends the proof.
\end{proof}

\vfill\eject

\section{Examples and extensions of main theorems}\label{s6}
In this section we discuss some examples where the results stated in the introduction apply.
We consider separately the case $\nabla u_0(P)\ne0$ and $\nabla u_0(P)=0$.
\subsection{Case $1:\nabla u_0(P)\ne0$}
\begin{ese}\label{D2}
Let $\O\subset\R^N$ be symmetric and convex with respect $x_1,..,x_N$ with $N\ge2$, $P\neq 0$  and $u_\e$ solution of
 \begin{equation}\label{I9}
\begin{cases}
-\Delta u=u^p~&\mbox{in}\ \O_\varepsilon,\\
u>0~&\mbox{in}\ \O_\varepsilon,\\
u=0~&\mbox{on}\ \partial\O_\varepsilon,
\end{cases}
\end{equation}
with $1<p<\frac{N+2}{N-2}$ for $N\ge3$ and $p>1$ if $N=2$. Moreover assume that
\begin{equation}\label{I10}
\int_{\O_\e}|\nabla u_\e|^2\le C,\quad\hbox{$C$ independent of $\e$}.
\end{equation}
Then $u_\e$ admits exactly $two$ critical points.
\end{ese}
\begin{proof}
Observe that \eqref{I10} is satisfied if  we consider a minimizer $u_\e$ of
\begin{equation}\label{D5}
\inf\limits_{u\in H^1_0(\O_\e),u\not\equiv0}\frac{\int_{\O_\e}|\nabla u|^2}{\left(\int_{\O_\e}|u|^{p+1}\right)^\frac2{p+1}}.
\end{equation}
By Corollary \ref{I11} it is enough to prove that $|u_\e|\le C$ in $\O_\e$ with $C$ independent of $\e$.
To do this we follow the line of the Gidas-Spruck proof in \cite{GS1981}. By contradiction suppose that $\|u_\e\|_\infty\to+\infty$
and let $x_\e$ be such that $\|u_\e\|_\infty=u_\e(x_\e)$ and
$v_\e:\|u_\e\|_\infty^\frac{p-1}2\big(\O_\e-x_\e\big)\to\R$ defined as
\begin{equation}\label{D6}
v_\e(x)=\frac1{\|u_\e\|_\infty}u_\e\left(x_\e+\frac x{\|u_\e\|_\infty^\frac{p-1}2}\right).
\end{equation}
Since $|v_\e|\le1$ it is immediate to check that $v_\e\to v$ in $C^2_{loc}(D)$ where $D$ is the limit domain of $\|u_\e\|_\infty^\frac{p-1}2\big(\O_\e-x_\e\big)$. Moreover $v$ satisfies
\begin{equation}\label{D4}
\begin{cases}
-\Delta v=v^p~&\mbox{in}\ D,\\
v>0~&\mbox{in}\ D,\\
v(0)=1,\\
v=0~&\mbox{on}\ \partial D.
\end{cases}
\end{equation}
We have that $D$ can be the whole space, a half-space or the exterior of a ball. The first two cases lead to a contradiction as in \cite{GS1981} because there is no solution to \eqref{D4}. Unfortunately we do not have a non-existence result for solutions to \eqref{D4} in the exterior of a ball and then the contradiction does not follow directly as before. On the other hand we have that by \eqref{I10} and $1<p<\frac{N+2}{N-2}$,
\begin{equation*}
\int_{\|u_\e\|_\infty^\frac{p-1}2\big(\O_\e-x_\e\big)}|\nabla v_\e|^2=\frac1{\|u_\e\|_\infty^{p+1-\frac{N(p-1)}2}}
\int_{\O_\e}|\nabla u_\e|^2\to0.
\end{equation*}
which gives that $v\equiv0$, a contradiction with $v(0)=1$. This proves that $|u_\e|\le C$ and it ends the proof.
\end{proof}
In the next example we remove the symmetry assumption replacing it with the condition that $p$ is close to $\frac{N+2}{N-2}$.

\begin{ese}
Let $\O\subset\R^N$, $N\ge3$ be convex and $u_\e$ solution of \eqref{I9} satisfying  \eqref{I10}. Then for $p$ sufficiently close to $\frac{N+2}{N-2}$ we have that
 $u_\e$ admits exactly $two$ critical points for $\e$ small enough.
\begin{proof}
In \cite{GM2003} it was showed that the solutions minimizing \eqref{D5} for $p=\frac{N+2}{N-2}-\d$
 admits a unique critical point (its maximum) if $\O$ is convex and $\d$ is small enough. Let us show that the maximum point $x_\d$ is nondegenerate. This is a consequence of the classical blow-up argument where is proved that the function $v_\d$ defined in \eqref{D6} satisfies
\begin{equation*}
v_\d(x)\to\frac{[N(N-2)]^\frac{N-2}4}{(1+|x|^2)^\frac{N-2}2}\,\,\quad\hbox{in }C^2\big(B(0,1)\big),
\end{equation*}
as $\d\to0$. Hence
$$\frac1{\|u_\d\|_\infty^p}\frac{\partial^2u_\d}{\partial x_i\partial x_j}(x_\d)=\frac{\partial^2v_\d}{\partial x_i\partial x_j}(0)=C(N)\d_{ij}+o(1),$$
which proves the nondegeneracy of the maximum point $x_\d$.

Next let us fix $\d$ small such that the previous properties hold and consider a solution $u_\e$ of \eqref{I9}. Then using the result of the previous example, we get that \eqref{2019-11-25-46} holds and Theorem \ref{th1-1} implies that $u_\e$ has exactly two critical points.
\end{proof}
\end{ese}
The last example is concerned with  semi-stable solutions as in Cabr\'e-Chanillo setting.

\begin{ese}\label{D1}
Let $\O\subset\R^2$ be a smooth bounded domain whose boundary  has positive curvature and $u_\e$ semi-stable positive solution to
 \begin{equation}\label{D8}
\begin{cases}
-\Delta u=\lambda f(u)~&\mbox{in}\ \O_\varepsilon,\\
u=0~&\mbox{on}\ \partial\O_\varepsilon,
\end{cases}
\end{equation}
where $f>0$ is an increasing function. Then for any $0<\lambda<\lambda^*$ and $\e$ small enough we have that $u_\e$ admits exactly $two$ critical points.
\begin{proof}
If we consider the problem
 \begin{equation}\label{D9}
\begin{cases}
-\Delta u=\lambda f(u)~&\mbox{in}\ \O,\\
u=0~&\mbox{on}\ \partial\O,
\end{cases}
\end{equation}
 it is known that there exists $\lambda^*>0$ such that for any  $0<\lambda<\lambda^*$ there exists a semi-stable solution to \eqref{D9}. Let us fix such a $\lambda$ and
consider a solution $u_\e$ to \eqref{D8}. It was proved in \cite{S}, pages $28$--$29$, that $\underline{u}=0$ is a subsolution and $\overline{u}=\a\left(\frac1{4D^2}-|x|^2\right)$ is a supersolution. Here $D$ is chosen such that $\O\subset B\left(0,\frac D2\right)$ and $\a$ (independent of $\e$) is properly chosen. So  \eqref{D8} admits a solution $u_\e$ which verifies $|u_\e|\le C$ with $C$ independent of $\e$.  Finally we have that by \cite{CC1998} the assumptions in
Theorem \ref{th1-1} are satisfied and then we get that $u_\e$ has exactly two critical points.
\end{proof}
\end{ese}
\begin{rem}
Note that the case of the first eigenfunction of $-\Delta$ falls in this last example.
\end{rem}
\subsection{Case $2:\nabla u_0(P)=0$}\label{D3}~

\vskip 0.2cm

First let us discuss some examples which satisfy the condition \eqref{2019-12-26-01}.
\begin{itemize}
\item $f(u)\equiv 1$. This is the well known torsion problem.  Here \eqref{2019-12-26-01} holds directly.
\item $\O$ convex and symmetric with respect to $P=0$ as in the Gidas, Ni and Nirenberg Theorem.
Observe that similarly as in  Lemma 2.1 in \cite{G2002} it was proved that in this case $\frac{\partial G(0,y)}{\partial x_i}$ and $\frac{\partial H(0,y)}{\partial x_i}$ are odd  with respect to $y_i$ for any $i=1,\cdots,N$. Let us prove that \eqref{2019-12-26-01} holds. Assuming that $u_\e$ is a solution to \eqref{aa2} which  verifies \eqref{2019-11-25-46} and $u_0$ its weak limit we have
 \begin{equation}\label{ad}
    \begin{split}
\int_{\Omega_\varepsilon}&\Big(f\big(u_\varepsilon(y)\big)-f\big(u_0(y)\big)\Big)
\frac{\partial G(x,y)}{\partial x_i}dy \\
=&\int_{\Omega_\varepsilon} \Big(f\big(u_\varepsilon(y)\big)-f\big(u_0(y)\big)\Big)
\left( \frac{\partial S(x,y)}{\partial x_i}+\frac{\partial H(0,y)}{\partial x_i}+\sum_{j=1}^N\frac{\partial^2 H(\xi,y)}{\partial x_i\partial x_j}x_j\right)dy
\\
=&\left(\hbox{using the oddness of }\frac{\partial H(0,y)}{\partial x_i} \right)
\\
=&o\big(|x|\big)- \int_{\Omega_\varepsilon} \Big(f\big(u_\varepsilon(y)\big)-f\big(u_0(y)\big)\Big)
\frac{\partial  S(x,y)}{ \partial y_i}  dy\\
=&
o\big(|x|\big)+ \int_{\Omega_\varepsilon} \frac{\partial \Big(f\big(u_\varepsilon(y)\big)-f\big(u_0(y)\big)\Big)}{\partial y_i}
 S(x,y)  dy +O \left(\int_{\partial B(0,\varepsilon)}\big|S(x,y)\big|d\sigma(y)\right)\\
 =&
o\big(|x|\big)+ O\left(\int_{\Omega_\varepsilon} \left|\frac{\partial \Big(f\big(u_\varepsilon(y)\big)-f\big(u_0(y)\big)\Big)}{\partial y_i}\right| \cdot \left| \sum_{j=1}^N\frac{\partial S(\xi,y)}{\partial x_j}x_i\right|
\right)dy\\&+O \left(\int_{\partial B(0,\varepsilon)}\big|S(x,y)\big|d\sigma(y)\right)\\
 =&
o\big(|x|\big) +
\begin{cases}
O\Big(\frac{\e^{N-1}}{|x|^{N-2}}\Big)&~\mbox{for}~N\geq 3,\\
O\Big(\e\cdot\big|\log |x|\big|\Big)&~\mbox{for}~N=2,
\end{cases}
\end{split}
\end{equation}
where $\xi$ is between $0$ and $x$.   Hence \eqref{2019-12-26-01} follows.
\end{itemize}
\begin{ese}
Let $\O\subset\R^N$ be symmetric and convex with respect $x_1,..,x_N$ with $N\ge2$, $P=0$  and $u_\e$ solution of \eqref{I9}
with $1<p<\frac{N+2}{N-2}$ for $N\ge3$ and $p>1$ if $N=2$. Moreover assume that
 \eqref{I10} holds. Then $u_\e$ admits at least $2N$ critical points.
\end{ese}
\begin{proof}
Similar to the proof of Example \ref{D2}, we have   that $|u_\e|\le C$ in $\O_\e$ with $C$ independent of $\e$. Also $0$ is the unique critical point of $u_0(x)$, which is a nondegenerate and  maximum point of $u_0(x)$. Then from \eqref{ad} and Theorem \ref{th1-2}, the claim follows.
\end{proof}
\begin{proof}[\textbf{Proof of Corollary \ref{cor1.7}}]
Set $u_\e(x)=u_\e(r)$with $r=|x|$. Since $f\ge0$ with $f\big(u_0(0)\big)>0$, integrating \eqref{aa2} we immediately get that $u_\e$ has a unique critical point $r_\e$.

By the discussion in Section \ref{D3} we have that \eqref{2019-12-26-01} holds and \eqref{C10} becomes (here $P=0$)
\begin{equation*}
\frac{u_\e'(r)}r=u_0''(0)+o(1)+
 \begin{cases}
(N-2)\big(u_0(0)+o(1)\big)\frac1{r^N}\e^{N-2}&~\mbox{if}~N\geq 3,\\
\big(u_0(0)+o(1)\big)\frac1{r^2}\frac{1}{|\log\e|}&~\mbox{if}~N=2,
\end{cases}
\end{equation*}
and then $u'(r_\e)=0$ if
\begin{equation*}
r_\e=
 \begin{cases}
\left[\left(\frac{N(N-2)u_0(0)}{f\big(u_0(0)\big)}\right)^\frac1N+o(1)\right]\e^\frac{N-2}N&~\mbox{if}~N\geq 3,\\
\left(\sqrt{\frac{u_0(0)}{2f\big(u_0(0)\big)}}+o(1)\right)\frac1{\sqrt{|\log\e|}}&~\mbox{if}~N=2.
\end{cases}
\end{equation*}
This ends the proof.
\end{proof}

\noindent\textbf{Acknowledgments} ~This work was done while Peng Luo  was visiting the Mathematics Department
of the University of Rome ``La Sapienza" whose members he would like to thank for their warm hospitality. Luo  was partially supported by NSFC grants (No.11701204,11831009) and the China Scholarship Council.

\end{document}